\documentclass[11pt,reqno,letterpaper]{amsart}
\usepackage[natbibapa]{apacite}
\RequirePackage[OT1]{fontenc}
\RequirePackage{amsthm,amsmath,graphicx,enumitem}
\RequirePackage[colorlinks,citecolor=blue,urlcolor=blue]{hyperref}
\usepackage{amssymb,tabularx,multicol,multirow,booktabs}
\usepackage[top=1in, bottom=1in, left=1in, right=1in]{geometry}
\usepackage{mhequ}
\usepackage[usenames,dvipsnames]{color}
\usepackage{tikz}
\usetikzlibrary{arrows.meta,positioning,patterns}
\usepackage{epstopdf}

\RequirePackage{xr}

\usepackage{xcolor}
\definecolor{darkbrown}{rgb}{0.5, 0.26, 0.13}

\numberwithin{equation}{section}
\theoremstyle{plain}
\newtheorem{theorem}{Theorem}[section]
\newtheorem{proposition}[theorem]{Proposition}
\newtheorem{lemma}[theorem]{Lemma}
\newtheorem{corollary}[theorem]{Corollary}
\theoremstyle{remark}
\newtheorem{remark}[theorem]{Remark}

\def \btheta{\boldsymbol{\theta}}
\def \bpi{\boldsymbol{\pi}}
\def \bp{\boldsymbol{p}}
\def \bn{\boldsymbol{n}}
\def \bN{\boldsymbol{N}}
\def \balpha{\boldsymbol{\alpha}}

\def \bX{\mathbf{X}}

\def \be{\begin{equs}}
	\def \ee{\end{equs}}

\def \E{\mathbb{E}}
\def \P{\mathbb{P}}

\def \R{\mathbb{R}}

\def \cR {\mathcal{R}}
\def \cC {\mathcal{C}}

\providecommand{\var}{\operatorname{Var}}
\providecommand{\bE}{\mathbf E}

\newcounter{cnstcnt}

\newcounter{cnst}

\begin{document}
	\title[Sharp risks for sparse submatrix detection]{Sharp minimax risks and phase transitions in sparse submatrix detection}
	
	\author{Subhajit Goswami}
	\address{School of Mathematics\\Tata Institute of Fundamental Research\\1, Homi Bhabha Road, Mumbai 400005, India}
	\email{goswami@math.tifr.res.in}
	
	\author{Rajarshi Mukherjee}
	\address{Harvard T.H. Chan School of Public Health\\ Department of Biostatistics\\
		655 Huntington Avenue, Boston, MA 02115}
	\email{rmukherj@hsph.harvard.edu}
		
		\makeatletter
	\@namedef{subjclassname@2020}{%
		\textup{2020} Mathematics Subject Classification}
	\makeatother
	
	\begin{abstract}
		In this paper, we study the minimax risk of detecting a sparse, elevated-mean Gaussian submatrix planted in a larger noisy matrix. In their seminal work, \cite{butuceasubmatrix2013} identified the sharp detection boundary for this problem, i.e., the signal profile across which the risk transitions between $0$ and $1$. We extend that zero-one theory by determining sharp risk asymptotics throughout a two-parameter phase diagram. 
		Our results are most naturally presented in the setting where the planted submatrix has size $n\times n$ and the ambient matrix has size $N\times N$, with $N=n^{1+\alpha}$ for some fixed $\alpha > 0$. In this framework, above the detection boundary, the risk decays stretched- or super-exponentially fast, and we determine its precise exponent. Below the boundary, where the risk converges to $1$, we obtain matching absolute-constant bounds for its polynomial rate of convergence. In both regimes, the asymptotics change across the line $\alpha+\delta=1/2$, where $\delta$ parametrizes the signed distance from the detection boundary. Finally, on the boundary, we show that the minimax risk converges to the nondegenerate constant $\frac12$ in a separate, very sparse regime where $n$ remains fixed and $N\to\infty$. 
		
		Each of these rates corresponds to the risk of a suitably calibrated scan or sum test, whence follow the upper bounds. The main problem is to derive 
		the matching regime-specific lower bounds. The mixture likelihood ratio is an average rather than a product and thus its logarithm lacks the additive 
		structure required by standard large deviation formulations. On the other hand, classical second-moment estimates fail to be sharp due to 
		overlap-induced correlations. To address these issues, we 
		develop two general lower-bound principles based on truncated mixture 
		likelihoods, tailored to the target risks on both sides of the detection boundary. The same framework yields sharp risk asymptotics for sparse 
		subtensors with equal side lengths and also applies to more general 
		testing problems including the detection of rectangular subtensors.
	\end{abstract}
	
	\maketitle

	\section{Introduction} \label{sec:intro}
	The theory of signal detection in high dimensions has witnessed substantial 
	progress in the last three decades. Building on the work of \cite{ingster2012nonparametric}, sharp asymptotic minimax separation results have been established for many problems of modern interest, including high-dimensional regression, large matrix and tensor models, and statistical network models. For a given instance of structured signal detection, this program aims to derive the minimum strength of a signal required to detect its presence against high-dimensional background noise as a function of the underlying structure and dimension. Mathematically, this corresponds to deriving a notion of critical signal 
	strength across which the minimax risk converges to $0$ or $1$; see, for instance, \cite{MR1456646}; \cite{donoho2004higher}; \cite{addarioberry2010combinatorial}; \cite{MR2747131}; \cite{ariascastro2011global}; \cite{ariascastro2011anomalous}; \cite{butuceasubmatrix2013}; \cite{MR4124336} and references therein. In 
	standard statistical parlance, this translates to characterizing regimes of 
	consistency or inconsistency of a testing problem under study.
	
	\vspace{0.05cm}
	
	However, classical statistics also enjoys a well-developed theory of asymptotic efficiency of tests, which goes beyond consistency and inconsistency to study a more precise asymptotic behavior of testing risks. Chernoff exponents \citep{chernoff1952measure} and Bahadur efficiency \citep{bahadur1967rates} 
	are classical manifestations of this viewpoint, providing finer measures of statistical difficulty than consistency alone. This paper develops an analogous sharp-risk theory for sparse Gaussian submatrix and subtensor detection. Our 
	starting point is the detection boundary of \cite{butuceasubmatrix2013} for 
	matrices, which we call the Butucea--Ingster detection boundary; see \eqref{def:Adelta}.
	
	\vspace{0.05cm}
	
	In the canonical square scaling, an elevated-mean $n\times n$ submatrix is 
	planted inside an $N\times N$ Gaussian noise matrix with $N=n^{1+\alpha}$. 
	Our results reveal a two-parameter phase diagram on both sides of the 
	Butucea--Ingster boundary; see Figure~\ref{fig:phase_diagram}. Above the boundary, the minimax risk is stretched- or 
	super-exponentially small in $n$, and we determine its principal exponential order. Below 
	the boundary, the risk tends to $1$ at a nontrivial power-law rate, which we determine up 
	to absolute multiplicative constants. Thus the smaller of the risk and its deficit from $1$ 
	changes from stretched- or super-exponential order above the boundary to polynomial 
	order below it. Within both half-planes, the form of the sharp asymptotics changes 
	across the secondary critical line $\alpha+\delta=1/2$, where $\delta$ measures the 
	signed distance from the detection boundary (see~\eqref{def:Adelta}). On the boundary, 
	in the very sparse regime where $n$ remains fixed and $N\to\infty$, we obtain the 
	non-degenerate limit $1/2$. 
	
	We extend these results to subtensors of every fixed order 
	$d\ge2$. Detection boundaries for such elevated-mean Gaussian subtensors were established by \cite{MR4382029}. When the planted and ambient tensors are $n \times \cdots \times n$ and $N \times \cdots \times N$ dimensional respectively, with $N = n^{1 + \alpha}$ as above, our Theorem~\ref{thm:tensor_extension} provides the corresponding sharp risk asymptotics displaying an analogous phase 
	diagram. This picture differs from the one arising in Gaussian sequence or other vector-type sparse testing problems, as illustrated,
	for instance, by the partial results and predictions of \cite{mukherjee2020minimax}. Thus the matrix case may be viewed as the first 
	non-vector member of a broader family of sparse Gaussian detection problems 
	whose alternatives are parametrized by product sets. To the best of 
	our knowledge, the current work gives the first {\em complete} above- and 
	below-boundary phase diagram for sharp minimax risk asymptotics in this family.

		Our method also yields sharp minimax rates in the more general setting where planted side lengths are distinct yet mutually comparable, while ambient side lengths scale with different polynomial exponents. Such 
		regimes are covered by \cite{butuceasubmatrix2013} and \cite{MR4382029}. We state our results for square matrices and tensors with equal side lengths, since these settings already exhibit the phases 
		in the corresponding rectangular problems, while making the underlying rates and proof mechanisms most transparent. 
		See the last paragraph in \S\ref{subsec:tensor_summary} for further 
		discussion.

	\vspace{0.07cm}

	Our upper bounds stem from appropriately calibrated sum and scan tests; 
	see \eqref{def:sumscantest}. Similar statistics were used in the 
	detection-boundary analyses of \cite{butuceasubmatrix2013} and 
	\cite{MR4382029}. The main challenge, however, lies in proving the optimality 
	of these bounds. We establish them using a general framework based on truncated mixture likelihoods which is a key contribution of the current work. We elaborate on this method in \S\ref{subsec:overview} after introducing the models and summarizing our main results.

	\subsection{Model, notation and statement of the testing problem}\label{sec:stro}
	We observe an $N \times N$ matrix $\bX=(X_{ij})_{1\le i, j \le N}$ with independent 
	entries
	\begin{equation}\label{def:Xdist}
		X_{ij}\sim N(\theta_{ij},1),
	\end{equation}
	where $\theta_{ij} \in \R$. For integers $1\le n\le N$ and signal strength $A > 0$, we 
	are interested in the following parameter space for testing against the simple null $\btheta = (\theta_{ij})_{1 \le i, j \le N} = \mathbf{0}$:
	\begin{equation}\label{def:Theta}
		\Theta(A,n, N) \stackrel{{\rm def.}}{=} \left\{
		\begin{array}{c}
			\btheta \in \R^{N \times N} : \exists\ \cR, \cC \subset \{1, \ldots, N\} \mbox{ with }|\cR|= |\cC| = n \mbox{ s.t. } \theta_{ij}\ge A \\ \mbox{ for } (i, j) \in \cR\times\cC \mbox{ and } \theta_{ij} = 0 \mbox{ otherwise }
		\end{array}
		\right\}.
	\end{equation}
     The square $n \times n$ submatrix setup 
	encoded in this definition is canonical in sparse submatrix detection and 
	appears widely in related work on planted submatrix localization, recovery, and computational barriers; see, e.g., \cite{ma2015computational}; \cite{MR3663635}; \cite{hajek2018submatrix}; \cite{banks2018information}; \cite{brennan2019universality}; \cite{gamarnik2021overlap}. Now given any pair of 
	sequences $A(N), n(N)$ indexed by $N$ with $N \to \infty$, we can consider the sequence of hypothesis 
	testing problems
	\begin{equation}\label{state:test_prob}
		H_0:\btheta\equiv \mathbf 0
		\qquad\mbox{against}\qquad
		H_1:\btheta\in\Theta(A(N), n(N), N).
	\end{equation}
	For a sequence of tests, i.e., $\{0, 1\}$-valued measurable functions $T_{N} = T_{N}(\bX)$ of $\bX$, we define the {\em maximum risk} over $\Theta(A(N), n(N), N)$ as
	\begin{equation}\label{eqn:bayes_risk}
		\mathrm{Risk}(T_N, A(N), n(N), N)\stackrel{}{=}
		\P_{0}(T_N = 1)+\sup_{\btheta\in\Theta(A(N), n(N), N)} \P_{\btheta}(T_N = 0),
	\end{equation}
	where $\P_{\btheta}$ (respectively $\P_0$) denotes the joint law in 
	\eqref{def:Xdist} with mean-matrix $\btheta$ (respectively $\mathbf 0$). The 
	corresponding {\em minimax risk} is
	\begin{equation}\label{def:minmax_risk}
		\mathcal R(A(N), n(N), N) = \inf_{T_N} \mathrm{Risk}(T_N, A(N), n(N), N)
	\end{equation}
	where the infimum is taken over all test sequences $\{T_N\}$.

	The square-matrix asymptotic setup underlying 
	Theorems~\ref{thm:dense_above_boundary}--\ref{thm:sparse_below_boundary} is
	\begin{equation}\label{eq:mnregime}
		N=\lfloor n^{1+\alpha}\rfloor, \qquad \alpha\in(0,\infty).
	\end{equation}
	Alternatively, we could also let $N$ range over all sufficiently large integers and 
	set $n = \lceil N^{1/(1+\alpha)}\rceil$. Now for any $\alpha\in(0,\infty)$ and 
	$\delta\in(-1,\infty)$, define
	\begin{equation}\label{def:Adelta}
		A^\ast(\delta, \alpha) = A^\ast(\delta, \alpha; n) =
		\min\left\{n^{-(1-\alpha-\delta)}, 
		\sqrt{\frac{4(1+\delta)\alpha\log n}{n}}\right\}.
	\end{equation}
	The Butucea-Ingster detection boundary \citep[Theorems~2.1--2.2]{butuceasubmatrix2013} may then be stated as
	\begin{equation}\label{eq:BIbndry0}
		\mathcal R_n(A^\ast(\delta, \alpha)) \stackrel{{\rm def.}}{=} \mathcal R(A^\ast(\delta, \alpha; n), n, \lfloor n^{1+\alpha}\rfloor) \xrightarrow{n \to \infty}
		\begin{cases}
			0, & \mbox{when } \delta > 0,\\
			1, & \mbox{when } \delta < 0,
		\end{cases}
	\end{equation}
	for fixed $\alpha$ and $\delta$. In this parametrization, $\delta$ is the signed distance 
	from the detection boundary and, together with $\alpha$, forms the two-parameter 
	coordinate system for the phase diagram (see Figure~\ref{fig:phase_diagram}). The present paper determines the risk $\mathcal 
	R_n(A^\ast(\delta, \alpha))$ to principal exponential order when $\delta > 0$ and 
	the difference $1 -\mathcal R_n(A^\ast(\delta, \alpha))$ up to absolute constants 
	when $\delta < 0$. We also treat a regime in Theorem~\ref{thm:criticality} where 
	$n$ is fixed and $N\to\infty$, showing that the minimax risk then converges to a 
	nondegenerate limit. In \S\ref{subsec:tensor_summary} we introduce a second asymptotic framework for subtensors and collect the extensions of all the above results to this setup in Theorem~\ref{thm:tensor_extension}. 
	We first summarize the square-matrix phase diagram, which is the principal narrative of the paper.

	\subsection{Summary of main results for 
		submatrices}\label{subsec:summary} 
	The main results are primarily organized according to the two regimes $\delta > 0$ and $\delta < 0$. The asymptotic behavior of the risk further 
	changes in each of these `off-boundary' regimes around the secondary critical line $\alpha+\delta = 1/2$. In view of the dichotomy apparent from \eqref{def:Adelta}, we call the sub-regimes characterized by $\alpha + \delta \le 1/2$ and $\alpha + \delta > 1/2$ as {\em dense} and {\em 
		sparse} respectively in the remainder of the article. These terms loosely refer to the size of $n$ relative to $N$ and thus degree of sparsity of problem when $\delta = 
	0$. Outside the regime \eqref{eq:mnregime}, in a fixed-submatrix-size setting, we obtain a non-trivial limiting risk in the `critical' case characterized by $\delta = 0$. We point the reader to Figure~\ref{fig:phase_diagram} for a schematic view of the square-matrix phase diagram. The precise square-case statements are 
	given in Theorems~\ref{thm:dense_above_boundary}--\ref{thm:criticality}; their tensor extension is Theorem~\ref{thm:tensor_extension}. 
	\begin{figure}[t]
		\begin{center}
			\resizebox{0.98\textwidth}{!}{
				\begin{tikzpicture}[x=10.8cm,y=6.1cm,>=Latex,font=\scriptsize]
					\def\amax{1.08}
					\def\ymax{0.4}
					\def\ymin{-0.42}
					\fill[gray!15]  (0,0) -- (0,\ymax) -- (0.12,\ymax) -- (0.50,0) -- cycle;
					\fill[gray!6] (0.12,\ymax) -- (\amax,\ymax) -- (\amax, 0) -- (0.50,0) -- cycle;
					\fill[gray!32] (0,\ymin) -- (0,0) -- (0.50,0) -- (0.92,\ymin) -- cycle;
					\fill[gray!43] (0.50,0) -- (\amax,0) -- (\amax,\ymin) -- (0.92,\ymin) -- cycle;
					\draw[->,thick] (0,\ymin) -- (0,0.45) node[above] {$\delta$};
					\draw[->,thick] (0,0) -- (1.14,0) node[right] {$\alpha$};
					\draw[very thick] (0,0) -- (\amax,0);
					\draw[dashed,thick] (0.50,\ymin) -- (0.50,\ymax);
					\draw[dotted,very thick] (0.1,\ymax) -- (0.50,0) -- (0.92,\ymin);
					\node[align=center,text width=2.8cm] at (0.2,0.12)
					{$\displaystyle \log\mathcal R_n(A^\ast(\delta, \alpha))$\\[-0.4pt]
						$\displaystyle \sim -\tfrac18 n^{2\delta}$};
					\node[align=center,text width=3.2cm, text = darkbrown, scale = 0.8] at (0.36,0.32)
					{$\displaystyle \log\mathcal R_n(A^\ast_{\bullet}(\delta, \alpha))$\\[0.1pt]
						\hspace{1.5pt} $\displaystyle \sim-\tfrac18 n^{2(\alpha+\delta)}$};
					\node[align=center,text width=4.5cm] at (0.79,0.22)
					{$\displaystyle \log\mathcal R_n(A^\ast(\delta, \alpha))$\\[-1pt]
						$\displaystyle \sim -\tfrac{\delta^2}{2(1+\delta)}\alpha n \log n$};
					\node[align=center,text width=3.1cm] at (0.2,-0.24)
					{$\displaystyle 1-\mathcal R_n(A^\ast(\delta, \alpha)) \asymp n^\delta$};
					\node[align=center,text width=5.1cm] at (0.85,-0.13)
					{$\displaystyle 1-\mathcal R_n(A^\ast(\delta, \alpha))$\\[-1pt]
						$\displaystyle \asymp \sqrt{(1+\delta)\alpha}\,n^{\frac12-\alpha}\sqrt{\log n}$};
					\node[align=center,text width=5.1cm, text = darkbrown, scale=0.8] at (0.65,-0.3)
					{$\displaystyle 1-\mathcal R_n(A^\ast_{\bullet}(\delta, \alpha))$\\[0.1pt]
						$\displaystyle \asymp \sqrt{(1+\delta)\alpha}\,n^{\frac12-\alpha}\sqrt{\log n}$};
					\node[above right] at (0.50,0.01) {$\alpha = 1/2$};
					\node[above,rotate=-30,fill=white,inner sep=1pt] at (0.35,0.15) {$\alpha+\delta=1/2$};
					\node[above] at (1.00,0.015) {$\delta=0$};
					\node[left] at (0,0) {$0$};
					\node[left] at (0,\ymax) {$\delta>0$};
					\node[left] at (0,\ymin) {$\delta<0$};
					\node[below right, scale=0.75] at (1,-0.01) {$\mathcal R_n(A_{\rm c}^\ast)\to \frac12$};
				\end{tikzpicture}%
			}
		\end{center}
		\caption{{\bf Phase diagram for the sharp minimax risk in the square-matrix model.} The horizontal line is the Butucea--Ingster detection boundary. Above it, the displayed rates (in black) describe the decay of $\mathcal R_n(A^\ast(\delta, \alpha))$; below it, they describe the difference $1-\mathcal R_n(A^\ast(\delta, \alpha))$. In each of these half-planes, the asymptotic form of
			$\mathcal R_n(A^\ast(\delta, \alpha))$ changes around the dotted line $\alpha+\delta=1/2$.
			Thus the lines $\delta = 0$ and $\alpha + \delta = 1/2$ split the diagram into four distinct phases for the decay of $\min\{\mathcal R_n(A^\ast(\delta, \alpha)), 1 - \mathcal R_n(A^\ast(\delta, \alpha))\}$ which are color coded in grayscale with darker shade indicating a larger asymptotic order. The rates corresponding to the signal strength $A_{\bullet}^\ast(\delta, \alpha)$ (see \eqref{def:AdeltaBI}) are displayed in {\color{darkbrown}brown} in the two sectors bound between the lines $\alpha + \delta = 1/2$ and $\alpha = 1/2$ which also mark the region where $A^\ast(\delta, \alpha) \ne A_{\bullet}^\ast(\delta, \alpha)$.}
		\label{fig:phase_diagram}
	\end{figure}
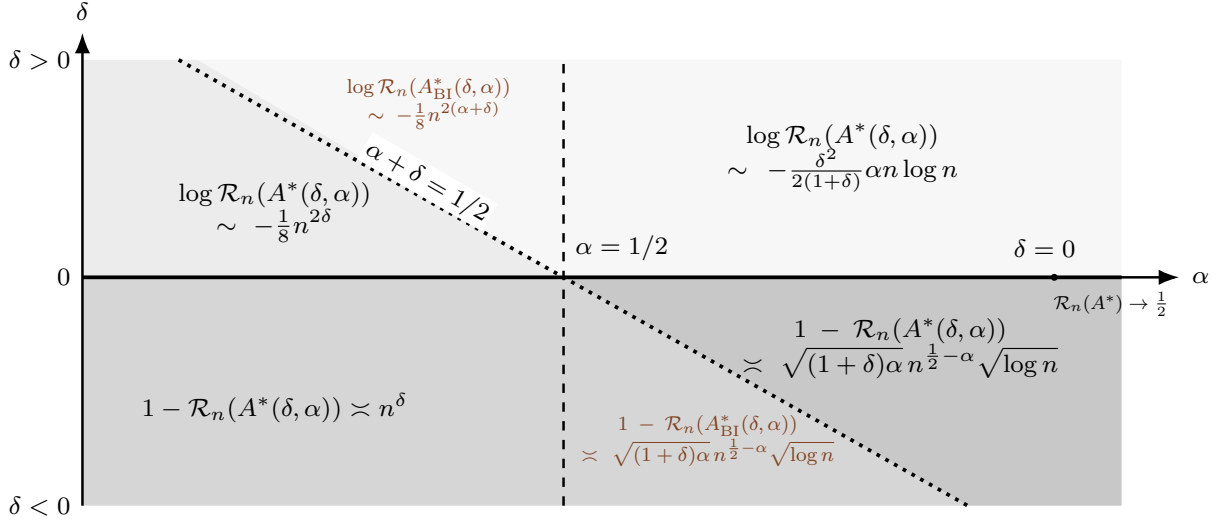
	
	\smallskip
	
	\noindent\emph{Above the detection boundary.}
	When $\delta>0$, the minimax risk tends to zero. In the dense part of the problem, $\alpha + \delta \le1/2$, a global sum test (see \eqref{def:sumscantest} below) is rate-optimal and it is shown in Theorem~\ref{thm:dense_above_boundary} that
	\begin{equation}\label{eq:densesupheu}
		\log\mathcal R_n(A^\ast(\delta, \alpha))\sim -\frac18 n^{2\delta}.    
	\end{equation}
	In the sparse sub-regime $\alpha+\delta > 1/2$, the so-called scan statistic (cf.~\eqref{def:sumscantest}) yields the optimal rate and we obtain in part~1 of Theorem~\ref{thm:sparse_riskbnd} that
	\begin{equation}\label{eq:sparsesupheu}
		\log\mathcal R_n(A^\ast(\delta, \alpha)) \sim -\frac{\delta^2}{4(1+\delta)} \log |\mathcal S|,
	\end{equation}
	where, with $[N] = \{1, \ldots, N\}$,
	\begin{equation}\label{def:Sfamily}
		\mathcal S = \mathcal S(n, N) \stackrel{{\rm def.}}{=}
		\{\cR\times\cC\subset[N]\times[N]: |\cR|=|\cC|=n\}
	\end{equation}
	denotes the set of all $n\times n$ submatrices of $[N]\times[N]$. Since $\log|\mathcal S|=(1+o(1))2\alpha n\log n$ under \eqref{eq:mnregime}, this gives the minimax risk to the principal exponential order.
	
	\vspace{0.1cm}
	
	In an alternative interpretation of the detection boundary adopted by 
	\cite{butuceasubmatrix2013} (see above Remark~2.1 in the paper), one can first 
	determine the active branch at $\delta = 0$ --- which changes its form according 
	to \eqref{def:Adelta} around the critical value $\alpha_c = 1/2$ --- and then vary 
	the signal strength within that branch. More precisely, with
	\begin{equation}\label{def:AdeltaBI}
		A^\ast_{\bullet}(\delta, \alpha) = A^\ast_{\bullet}(\delta, \alpha; n) \stackrel{{\rm def.}}{=} 
		\begin{cases}
			n^{-(1 - \alpha - \delta)}, & \alpha \le 1/2,\\
			\sqrt{\frac{4(1+\delta)\alpha\log n}{n}}, & \alpha > 1/2
		\end{cases} \qquad \mbox{for } \delta\in(-1,\infty),
	\end{equation}
	the results of \cite{butuceasubmatrix2013} yield that $$\mathcal R_n(A^\ast_{\bullet}(\delta, \alpha))  \to 0/1 \:\mbox{ for $\delta \gtrless 0$ respectively.}$$
	It is a natural question to ask how the 
	phase diagrams for these two signal strength profiles are related. Note that we can have 
	$A^\ast_{\bullet}(\delta, \alpha) \neq A^\ast(\delta, \alpha)$ for some $\delta > 0$ {\em only if} $\alpha \le 1/2 < \alpha + \delta$. In Figure~\ref{fig:phase_diagram}, this corresponds to the sector above the diagonal line $\alpha + \delta = 1/2$ in the second quadrant. It turns out 
	that a scan test is still rate-optimal in this sub-regime and we establish in part~2 of 
	Theorem~\ref{thm:sparse_riskbnd}:
	\begin{equation}\label{eq:densesupBIheu}
		\log\mathcal R_n(A^\ast_{\bullet}(\delta, \alpha))\sim -\frac18 n^{2(\alpha+\delta)}.
	\end{equation}

	\vspace{0.2cm}
	
	\noindent\emph{Below the detection boundary.}
	When $\delta < 0$, the minimax risk 
	tends to 1. Theorems~\ref{thm:dense_below_boundary} and \ref{thm:sparse_below_boundary} show 
	that the deficit from 1 is {\em polynomial} rather than (stretched or 
	super-)exponentially small. In the dense below-boundary (sub-)regime $\alpha+\delta\le1/2$, 
	a sum test statistic achieves the optimal rate like in the dense above-boundary regime 
	discussed above (but for a new threshold) and we show in 
	Theorem~\ref{thm:dense_below_boundary} that
	\begin{equation}\label{eq:densesubheu}
		1-\mathcal R_n(A^\ast(\delta, \alpha))\asymp n^{\delta},
	\end{equation}
	with absolute implicit constants. This would suggest that the optimal rate in the 
	sparse below-boundary regime $\alpha + \delta > 1/2$ should come from a scan 
	test as in the \sloppy sparse above-boundary regime. Interestingly, this heuristic 
	turns out to be {\em false} and the same sum test still attains the optimal rate which 
	we identify in Theorem~\ref{thm:sparse_below_boundary} as
	\begin{equation}\label{eq:sparsesubheu}
		1-\mathcal R_n(A^\ast(\delta, \alpha))
		\asymp \sqrt{(1+\delta)\alpha}\, n^{1/2-\alpha}\sqrt{\log n},
	\end{equation}
	again up to absolute multiplicative constants. In fact, \eqref{eq:sparsesubheu} 
	remains valid as long as the sparse signal profile $\sqrt{4(1+\delta)\alpha\log n / n}$ is used.  Thus it also covers the region where $A^\ast(\delta, \alpha) \neq A^\ast_{\bullet}(\delta, \alpha)$ for $\delta < 0$ (revisit \eqref{def:AdeltaBI} above and see Figure~\ref{fig:phase_diagram}).

	\vspace{0.2cm}
	
	\noindent\emph{On the boundary.}
	Finally, Theorem~\ref{thm:criticality} treats a fixed-$n$ very sparse critical regime 
	in which $N \to \infty$ and the signal strength is given by $$A_{\rm c}^\ast = A_{\rm c}^\ast(n , N) =  \sqrt{\frac{4\log N}n} \, ;$$  
	cf.~\eqref{def:Adelta} corresponding to $\delta = 0$ and note that $\alpha \log n = \log \frac Nn$ for $N = n^{1 + 
		\alpha}$. In this case, Theorem~\ref{thm:criticality} tells us that
	\begin{equation}\label{eq:sparsecritheu}    
		\lim_{N \to \infty} \mathcal R(A_{\rm c}^\ast, n, N) = \frac12 \:\:\: \mbox{(recall \eqref{def:minmax_risk} for notation)}.
	\end{equation}
	Thus the critical behavior is genuinely non-degenerate in this scaling. This 
	behavior is also markedly {\em different} from the sparse Gaussian 
	sequence model and the linear regression model \citep{ingster2012nonparametric, mukherjee2020minimax} where the minimax risk for detecting $s$-sparse signals at 
	criticality converges to $\frac1{2^s}$ for fixed $s$. Therefore, whereas the 
	minimax risk at criticality for sparse vector signal detection decreases with the 
	size of the (sparse) signal, the minimax risk at criticality for detecting fixed-sized submatrices does not depend on the submatrix size. We will see in 
	\eqref{eq:sparsecritheud} below that the same limiting value $\frac12$ holds 
	for the tensor version of the problem as well.

	\subsection{Subtensor extension}\label{subsec:tensor_summary}
	One can similarly consider the problem of detecting an $n \times \cdots \times n$ dimensional order-$d$ tensor embedded in a larger $N \times \ldots \times N$ tensor for any $d \ge 2$ ($d = 2$ being the matrix case). Formally, we observe an $N \times \ldots \times N$ tensor 
	$\bX=(X_{i_1\ldots i_d})_{1 \le i_1, \ldots, i_d \le N}$ with independent entries $X_{i_1\ldots i_d}\sim N(\theta_{i_1\ldots i_d},1)$ where 
	$\theta_{i_1\ldots i_d} \in \R$. The corresponding alternative parameter space is
	\begin{equation}\label{def:Theta_tensor}
		\begin{split}
			\Theta_d(A, n, N) \stackrel{{\rm def.}}{=} \left\{
			\begin{array}{c}
				\btheta \in \R^{N \times \cdots \times N} : \exists\ \mathcal I_i\subset[N]\ \mbox{ with }\ |\mathcal I_i|=n,\ 1\le i\le d \mbox{ s.t. } \\ \theta_{i_1\ldots i_d}\ge A \mbox{ for } (i_1, \ldots, i_d) \in \mathcal I_1\times\cdots\times\mathcal I_d \mbox{ and } \theta_{i_1\ldots i_d} = 0 \mbox{ otherwise }
			\end{array}
			\right\}
		\end{split}
	\end{equation}
	(cf.~\eqref{def:Theta}). We are interested in the hypothesis testing problems \eqref{state:test_prob} with $\Theta_d(A, n, N)$ in place of $\Theta(A,n,N)$. Now defining the maximum risk $\mathrm{Risk}_d(T_N, A(N), n, N)$ associated with a test sequence $T_N$ as in \eqref{eqn:bayes_risk} (with $\Theta_d(A, n, N)$ replacing $\Theta(A, n, N)$), let
	\begin{equation}\label{def:minmax_risk_tensor}
		\mathcal R_d(A, n, N)=\inf_T\mathrm{Risk}_d(T, A, n, N)
	\end{equation}
	denote the minimax risk where the minimum is over all possible tests. We 
	work under the same asymptotic setting \eqref{eq:mnregime} (except in \eqref{eq:sparsecritheud} below).

	The detection boundary in this setup was identified in \cite{MR4382029} as stated below:
	\begin{equation}\label{def:Adelta_tensor}
		A^\ast_d(\delta, \alpha) = A^\ast_d(\delta, \alpha; n)\stackrel{{\rm def.}}{=}
		\min\left\{n^{-\frac d2(1-\alpha-\delta)},
		\sqrt{\frac{2d(1+\delta)\alpha\log n}{n^{d-1}}}\right\},
		\qquad \delta\in(-1,\infty).
	\end{equation}
	In this case, the critical line demarcating the dense/sparse transition becomes 
	$\alpha + \delta = 1/d$. Our proofs along with the dichotomy between sum/scan mechanism for $d = 2$ extends, {\em mutatis mutandis}, to any 
	general $d$. Therefore, we only prove our main results in the upcoming section(s) for 
	matrices while collecting all the results for the general order-$d$ case inside a single theorem, namely Theorem~\ref{thm:tensor_extension}, in 
	\S\ref{subsec:tensor_extension}.

	In the dense above-boundary regime, i.e., for $\alpha + \delta \le 1/d$ and $\delta > 0$, the minimax risk behaves like 
	\begin{equation}\label{eq:densesupheud}
		\log\mathcal R_{d, n}(A^\ast_d(\delta, \alpha)) \sim -\frac18 n^{d\delta}
		\tag{\ref{eq:densesupheu}$'$}
	\end{equation}
	where, $\mathcal R_{d, n}(A) \stackrel{{\rm def.}}{=} \mathcal R_d(A, n, \lfloor n^{1+\alpha}\rfloor)$. On the other hand, in the sparse above-boundary 
	regime given by $\alpha + \delta > 1/d$ and $\delta > 0$, we have
	\begin{equation}\label{eq:sparsesupheud}
		\log\mathcal R_{d, n}(A^\ast_d(\delta, \alpha)) \sim -\frac{\delta^2}{4(1+\delta)} \log |\mathcal S_d|,    
		\tag{\ref{eq:sparsesupheu}$'$}
	\end{equation}
	where $\mathcal S_d$ denotes the set of all $n\times \cdots \times n$ order-$d$ 
	subtensors of $[N] \times \cdots \times [N]$. This is identical to the matrix, i.e., $d = 2$ case except that now we have $\log|\mathcal S_d|=(1+o(1))d\alpha n\log n$. 
	
	\smallskip
	
	We can also define a signal strength function $A^\ast_{d, \bullet}(\delta, \alpha)$ analogous to \eqref{def:AdeltaBI}, i.e., 
	\begin{equation}\label{def:AdeltadBI}
		A^\ast_{d, \bullet}(\delta, \alpha) = A^\ast_{d, \bullet}(\delta, \alpha; n) \stackrel{{\rm def.}}{=} 
		\begin{cases}
			n^{-\frac d2(1 - \alpha - \delta)}, & \alpha \le d/2,\\
			\sqrt{\frac{2d(1+\delta)\alpha\log n}{n^{d-1}}}, & \alpha > d/2
		\end{cases} \qquad \mbox{for } \delta\in(-1,\infty).
	\end{equation}
Similar to the matrix version (with 
	$d$ in place of 2), we have $A^\ast_{d, \bullet}(\delta, \alpha) \ne A^\ast_{d}(\delta, \alpha)$ for some $\delta > 0$ only if $\alpha + \delta > 1/d \ge \alpha$, in which case 
	\begin{equation}\label{eq:densesupBIheud}
		\log\mathcal R_{d, n}(A^\ast_{d, \bullet}(\delta, \alpha))\sim -\frac18 n^{d(\alpha+\delta)}.
		\tag{\ref{eq:densesupBIheu}$'$}
	\end{equation}
	
	\vspace{0.15cm}
	
	Below the detection boundary, the risks again exhibit power-law decay around the limiting 
	value 1. In the dense below-boundary regime $\alpha + \delta \le 1/d$, we get
	\begin{equation}\label{eq:densesubheud}
		1-\mathcal R_{d, n}(A^\ast_d(\delta, \alpha))\asymp n^{d\delta/2}
		\tag{\ref{eq:densesubheu}$'$}
	\end{equation}
	while for all $\alpha > 1/d$, we have
	\begin{equation}\label{eq:sparsesubheud}
		1-\mathcal R_{d, n}(A^\ast_d(\delta, \alpha))
		\asymp \sqrt{(1+\delta)\alpha}\, n^{1/2 - d\alpha/2}\sqrt{\log n}.
		\tag{\ref{eq:sparsesubheu}$'$}    
	\end{equation}
	The multiplicative constants implicit in ``$\asymp$'' involve $d$ for both the 
	bounds.
	
	\vspace{0.1cm}
	
	Finally, in the fixed-$n$ sparse critical regime, the signal strength is $A^\ast_{d, {\rm c}} = \sqrt{2d\log N/n^{d-1}}$ and 
	\begin{equation}\label{eq:sparsecritheud}    
		\lim_{N \to \infty} \mathcal R_d(A^\ast_{d, {\rm c}}, n, N) = \frac12.
		\tag{\ref{eq:sparsecritheu}$'$} 
	\end{equation}
	Thus the critical limiting minimax risk in the fixed-size setting neither depends on $n$ {\em nor} on the order $d$ of the tensor.

\medskip	
	
	Let us end the subsection with few comments on rectangular subtensor detection. Denoting 
	the vector of side lengths of the planted and ambient tensors by
	\begin{equation*}
		\bn=(n_1,\ldots,n_d)  \mbox{ and }  \bN=(N_1,\ldots,N_d).
	\end{equation*}
	respectively, consider the asymptotic setting
	\begin{equation}\label{eq:tensor_regime}
		N_i=\lfloor n_i^{1+\alpha_i}\rfloor,\qquad
		\min_{1\le i \le d}n_i\longrightarrow\infty,\qquad
		\kappa^{-1}\le\frac{n_i}{n_j}\le\kappa
		\quad(1\le i,j\le d)
	\end{equation}
	with $\balpha=(\alpha_1,\ldots,\alpha_d)\in(0,\infty)^d$ and $\kappa\ge1$ fixed. Detection boundary for rectangular Gaussian submatrices in this 
	regime was the subject of the original formulation by 
	\cite{butuceasubmatrix2013} (see assumptions (2.6)--(2.7)). The corresponding statements for sparse subtensor detection were established 
	in \cite{MR4382029} (see display~(9)). Although we skip a detailed 
	description of the sharp minimax rates in this case to keep our exposition 
	clear and focused, we give a one-line summary below. The rates as well as the accompanying phase transitions and the proofs from the square-matrix case carry over {\em directly} to the above setting by replacing 
	\begin{equation*}
		\text{\begin{minipage}{0.90\textwidth}	$2\alpha$ with $\alpha_\Sigma {=} \sum_{i=1}^d\alpha_i$, $n^2$ with $\prod_{i = 1}^d n_i$, $2\alpha n \log n$ with $\sum_{i = 1}^d \alpha_i n_i \log n_i$, and $1/2$ with $1/d$.\end{minipage}}
	\end{equation*}

	\subsection{Discussion of proof method}\label{subsec:overview}
	For $S\in\mathcal S$ (recall \eqref{def:Sfamily}) or, more generally, $S \subset [N] \times [N]$, write
	\begin{equation}\label{def:XS}
		X_S\stackrel{{\rm def.}}{=}\sum_{(i,j)\in S}X_{ij}.    
	\end{equation}
	The upper bounds arise from two test statistics, namely 
	\begin{equation}\label{def:sumscantest}
		T_{{\rm sum}}^H  \stackrel{{\rm def.}}{=} \mathbf{1}_{\{X_{[N]\times[N]} > H\}} \,\, \mbox{ and } \,\, T_{{\rm scan}}^H \stackrel{{\rm def.}}{=} \mathbf{1}_{\big\{\max\limits_{S\in \mathcal{S}} \frac{X_{S}}n > H \big\}}
	\end{equation}
	referred to as the {\em sum} and {\em scan} tests at threshold $H$ (with implicit dependence on $n$ and $N$). The same statistics were also used 
	for constructing `good' tests in \cite{butuceasubmatrix2013}; their tensor 
	analogues appear in \cite{MR4382029}. Once a statistic is selected, evaluating the sum of its type-I and worst-case type-II errors only requires standard Gaussian tail estimates and --- as in the case of scan --- a union 
	bound. We can then optimize this error with respect to $H$. We can generate reasonable upper bounds on the minimax risk in this manner with an important caveat that the choice between the two types of tests may 
	be different in the below-boundary regime (where all tests are `bad') from its above-boundary counterpart. The main and much 
	more difficult part is to show that these rates are optimal over all possible tests.

   \vspace{0.2cm}

We will focus our discussion of the lower bounds primarily on the 
above-boundary regime. Notice that if our goal were to establish a vanishing minimax risk above a certain 
threshold, as in the detection-boundary problems considered by 
\cite{butuceasubmatrix2013} and \cite{MR4382029}, it would suffice to exhibit a sequence of tests --- such as 
those defined in \eqref{def:sumscantest} --- whose risks converge to $0$. 
Consequently, quantitative lower bounds in this regime are 
relatively scarce for similar high-dimensional testing problems. Here, high-dimensionality manifests through the set $\Theta(A) = \Theta(A, n, N) \subset \R^{N \times N}$ (recall \eqref{def:Theta}) which parametrizes the space of alternative hypotheses. A standard approach to establishing lower bounds on the minimax risk in this setting is to first reduce to a {\em simple-versus-simple} testing 
problem by placing a suitable prior $\bp$ on $\Theta(A)$. Due to the obvious monotonicity of the probability measures indexed by $\Theta(A)$ with respect to 
the signal entries, it suffices for our purpose to use the uniform prior $\bpi$ on the boundary set:
	\begin{equation}\label{def:partialTheta}
		\partial \Theta(A) \stackrel{{\rm def.}}{=} \big\{\btheta:\exists\ \cR\subset[N],\ \cC\subset[N],\ |\cR| = n,\ |\cC| = n,
		\theta_{ij} = A\mathbf 1_{(i,j)\in \cR\times\cC}\big\}.
	\end{equation}
The corresponding minimum risk, i.e., the {\em Bayes risk} is given by $1 - d_{{\rm TV}}(\P_{\bpi}, \P_0)$ where $\P_{\bpi} \stackrel{{\rm def.}}{=} \int_{\Theta(A)} \P_{\btheta} d\bpi(\theta)$ denotes the mixture measure 
	under $\bpi$ and $d_{{\rm TV}}(\cdot, \cdot)$ is the total-variation 
	distance. The classical chi-square bound then yields (see, e.g., \cite[Section~2]{MR2724359}, \cite[Section~5.1]{silber2026testing}):
	\begin{equation}\label{eq:CS}
	\chi^2(\P_{\bpi} || \P_0 ) = \E_0( L_{\bpi}^2) - 1 \ge \max\left\{ \frac{1}{2 ( 1- d_{{\rm TV}}(\P_{\bpi}, \P_0))} - 1,  (2 d_{{\rm TV}}\big(\P_{\bpi}, \P_0)\big)^2 \right\}
\end{equation}	
where $L_{\bpi}$ denotes the integrated likelihood ratio under $\bpi$, defined 
as
\begin{equation}\label{def:Lpi}
	L_{\bpi}=\frac1{|\mathcal S|}\sum_{S\in\mathcal S}
	\exp\left(A X_S-\frac{A^2 n^2}{2}\right).
\end{equation}
	Usually, $\chi^2(\P_{\bpi} || \P_0)$ diverges in the above-boundary regime, 
rendering the second inequality in \eqref{eq:CS} uninformative. On the other 
hand, lower bounds derived from the first inequality are typically very weak 
compared to the upper bounds obtained via the arguments in the first 
paragraph. 

\vspace{0.03cm}

Another plausible route is to directly 
estimate the probability of the event $\{L_{\bpi} > 1\}$ under either 
$\P_0$ or $\P_{\bpi}$, which --- by the Neyman-Pearson lemma --- add 
up to the Bayes risk. However, the lack of an additive structure in the 
log-likelihood ratio and the nature of our risk bounds rule out any classical 
large deviation formulation. See \S\ref{subsec:connection} below for more 
details.

\vspace{0.2cm}

\noindent{{\em Truncated likelihood ratios.}}  Events associated with establishing upper bounds, such as those underlying the statistics $T^{H}_{{\rm sum}}$ and $T^{H}_{{\rm scan}}$, often provide `plausible' ways for deducing matching lower bounds. A precise way to implement this is via {\em truncation} with respect 
to such events. Truncated second-moment analysis is quite common in the 
high-dimensional testing for demonstrating the impossibility of detection and was already employed in \cite{butuceasubmatrix2013, MR4382029} for this purpose. 
In those settings, truncation is primarily used to tame the correlations in the second moment, with the truncating events holding with probability close to 1 so 
that their overall effect on the risk remains negligible. 

In this work, we take a distinct approach: we use truncated analogues of the inequalities in \eqref{eq:CS} that permit truncation with respect to events whose 
probabilities match the {\em target} risk bound, which may be extremely small (as in the above-boundary regime). We informally refer to this technique as {\em pushing}. Formally, for any family of truncating events $\bE$ parametrized by $\Theta(A)$ (see Section~\ref{sec:prelim} for the precise definitions), we bound 
the Bayes risk from below in Proposition~\ref{prop:risk_lowerbnd_ab} by
\begin{equation}\label{eq:ineq1}
\min\left\{\frac12\E_0(L_{\bpi}^{\bE}),
\frac{(\E_0(L_{\bpi}^{\bE}))^2}{4\E_0((L_{\bpi}^{\bE})^2)}\right\}
\end{equation}
(cf.~the first inequality in \eqref{eq:CS}), where $L_{\bpi}^{\bE}$ is a truncated version of the likelihood ratio $L_{\bpi}$. In 
Proposition~\ref{prop:risk_lowerbnd_bb}, we provide the complementary bound
\begin{equation}\label{eq:ineq2}
\E_0(L_{\bpi}^{\bE})  - \frac12\sqrt{\var_0(L_{\bpi}^{\bE})}	
\end{equation}
(cf.~the second inequality in \eqref{eq:CS}). 
In applications, the first moment term $\E_0(L_{\bpi}^{\bE})$ represents the target rate. Indeed, in view of the identity
$$\E_0(L_{\bpi}^{\bE}) = \int\P_{\btheta}(\bE(\btheta))\,d\bpi(\btheta),$$
we see that $\E_0(L_{\bpi}^{\bE})$ is the average probability of the truncating event $\bE(\btheta)$ under the alternative $\P_{\btheta}$, which guides our choice of $\bE$ (see Remark~\ref{rm:bE}).

Together, the inequalities \eqref{eq:ineq1} and \eqref{eq:ineq2} provide a general and flexible framework for obtaining sharp minimax risk lower bounds across  
a broad class of testing problems, including the rectangular subtensor detection problem (see~\eqref{eq:tensor_regime}). Their flexibility lies in the freedom to `push' different parametrized event families $\bE$ tailored to the problem in hand. Consequently, our choice of $\bE$ adapts to each of the sub-regimes depicted in 
Figure~\ref{fig:phase_diagram}. We discuss further potential 
applications of this framework in \S\ref{subsec:directions} below.

\vspace{0.1cm}
	
Before concluding this subsection, we briefly comment on the second moment of $L_{\bpi}^{\bE}$ that appears in both inequalities. It takes the generic form:
\begin{equation*}
	\frac1{|\mathcal S|} \sum_{0 \le k, \ell \le n} \exp( q(k\ell) ) \, N_{k, \ell}
\end{equation*}
where $N_{k,\ell}$ denotes the number of submatrices having $k$ common rows 
and $\ell$ common columns with a fixed submatrix, and $q(k\ell) = 
q(k\ell; A, n)$ captures correlations arising from submatrix overlaps. We can interpret this expression as a (finite-volume) partition 
function over overlap patterns in which $q(k\ell)$ acts as the energy term 
while $\log N_{k,\ell}-\log|\mathcal S|$ is its entropy cost. {\em Above} the 
boundary, the desired large-deviation-type bound is determined by locating 
the leading free-energy contribution. {\em Below} the boundary, by contrast, 
we need a polynomially decaying estimate on the variance which involves 
controlling contributions from multiple overlap ranges rather than only the maximizing one.
		
	\subsection{Organization} In Section~\ref{sec:main_results}, we state the square-matrix theorems followed by their subtensor extension. 
	Section~\ref{sec:discussions} gives further discussion, including 
	connections with classical and recent literature on hypothesis testing, 
	computational considerations, and open directions. Section~\ref{sec:prelim} 
	formulates two general lower-bounds on the minimax risk, respectively tailored to the above-boundary and below-boundary regimes, which provide 
	the common technical framework for the proofs. It also collects some useful properties of the hypergeometric and Gaussian distributions. 
	Sections~\ref{sec:dense_above}--\ref{sec:criticality} are devoted to the proofs of the square-matrix results. 
	
	\subsection{Notation} \label{subsec:notation}
	We write $c,c',C,C', \ldots$ etc. for finite, positive constants whose values may change from line to line. Unless indicated otherwise, constants are absolute; dependence on parameters such as $\alpha$, $\delta$ etc. is displayed when relevant. For brevity, we suppress the dependence of the parameter space $\Theta(A, n, N)$ and the submatrix family $\mathcal S(n, N)$ on $n$ and $N$. Similarly, we omit $n$ from the notations $A^\ast(\delta, \alpha; n)$ and $A^\ast_{\bullet}(\delta, \alpha; n)$. The same conventions also apply to the corresponding tensor notations introduced in \S\ref{subsec:tensor_summary} like $A^\ast_{d}(\delta, \alpha; n)$ etc. We use $[M]$ to denote the set of integers $\{1,\ldots, M\}$ 
	for any positive integer $M$. The joint distribution of $\bX$ under the mean matrix is denoted by $\P_{\btheta}$, and the expectation and 
	variance with respect to $\P_{\btheta}$ are denoted by $\E_{\btheta}$ and $ 
	\mathrm{Var}_{\btheta}$, respectively. Under the null mean 
	$\btheta=\mathbf0$, we write $\P_0,\E_0$, and ${\rm Var}_0$.

	\section{Main Results} \label{sec:main_results}
	The square-matrix model remains the primary setting for the exposition of our main results and their proofs. In \S\ref{subsec:square_matrix}, we state separate theorems characterizing each of the phases depicted in Figure~\ref{fig:phase_diagram}. They 
	are further grouped into three subsections covering respectively the above-boundary, 
	below-boundary and critical regimes. Their extensions to the tensor model with equal side lengths is collected in \S\ref{subsec:tensor_extension} 
	as Theorem~\ref{thm:tensor_extension}.
	
	\subsection{Results for square submatrix detection}\label{subsec:square_matrix}
	Except for Theorem~\ref{thm:criticality}, all the results in this subsection take place 
	in the asymptotic framework \eqref{eq:mnregime}.
	\subsubsection{Above-boundary regime $(\delta>0)$}\label{subsec:main_above_bnd}
	The first theorem concerns the dense side of the phase diagram where $\alpha + \delta \le 1/2$. In this part, the signal 
	strength $A^\ast(\delta, \alpha)$ as given by \eqref{def:Adelta} is the algebraic rate $n^{-(1-\alpha-\delta)}$.
	\begin{theorem}[Dense above-boundary regime]\label{thm:dense_above_boundary}
		Let $\alpha + \delta \leq 1/2$ and $\delta>0$. 
		Then
		\begin{equation}\label{eq:dense_riskbnd1}
			\lim_{n\to\infty}
			\frac{\log \mathcal R_n(A^\ast(\delta, \alpha))}{n^{2\delta}}
			=-\frac18.
		\end{equation}
	\end{theorem}

	The next theorem covers the sparse scan-dominated sub-regime. We state it for 
	$\alpha + \delta > 1/2$, where this is the Butucea--Ingster sparse regime, but the 
	proof gives the same exponent for every $\alpha > 0$ if the sparse regime signal 
	strength 
	\begin{equation}\label{def:sparse_profile}
		A^\ast_{\rm S}(\delta, \alpha) = A^\ast_{\rm S}(\delta, \alpha; n) \stackrel{{\rm def.}}{=} \sqrt{\frac{4(1+\delta)\alpha\log n}{n}}
	\end{equation}
	(cf.~\eqref{def:Adelta}) is used. We also give the asymptotic rate corresponding to 
	the signal strength $A^\ast_{\bullet}(\delta, \alpha) = n^{-(1 - \alpha -\delta)}$ for $\alpha + \delta > 1/2 \ge \alpha$ which defines the region where $A^\ast_{\bullet}(\delta, \alpha)$ and $A^\ast(\delta, \alpha)$ are not equal for $\delta > 0$ (see \eqref{def:AdeltaBI} and the surrounding discussion in \S\ref{subsec:summary}). 
	\begin{theorem}[Sparse above-boundary regime]\label{thm:sparse_riskbnd}
		Let $\alpha + \delta > 1/2$ and $\delta > 0$.
		\begin{enumerate}
			\item We have
			\begin{equation}\label{eq:sparse_riskbnd}
				\lim_{n\to\infty}
				\frac{\log\mathcal R_n(A^\ast(\delta, \alpha))}{\log|\mathcal S|}
				=-\frac{\delta^2}{4(1+\delta)}
			\end{equation}
			(recall $\mathcal S$ from \eqref{def:Sfamily}) and the same conclusion holds for all $\alpha > 0$ corresponding to the signal strength $A^\ast_{\rm S}(\delta, \alpha)$.
			\item If also $\alpha \le 1/2$, then
			\begin{equation}\label{eq:dense_riskbnd2}
				\lim_{n\to\infty}
				\frac{\log \mathcal R_n(A^\ast_{\bullet}(\delta, \alpha))}{n^{2(\alpha+\delta)}}
				=-\frac18.
			\end{equation}
		\end{enumerate}
	\end{theorem}
	
	Since $\log|\mathcal S|=(1+o(1)) \, 2\alpha n\log n$ (see \eqref{eq:Scardinality} below), part~1 of Theorem~\ref{thm:sparse_riskbnd} determines the minimax risk up 
	to the principal exponential order. The proof of part~2 is essentially a 
	corollary of the proof of part~1.

	\subsubsection{Below-boundary regime $(\delta < 0)$}\label{subsec:main_below_bnd}
	Below the detection boundary, the minimax risk tends to 1. The following two 
	theorems quantify the size of the quantity $1-\mathcal R_n(A^\ast(\delta, \alpha))$. 
	The first result below characterizes the case when $\alpha+\delta\leq\frac{1}{2}$ --- 
	which can arise when $\alpha\leq \frac{1}{2}$ or $\alpha>\frac{1}{2}$ but $\delta$ negative enough. 
	\begin{theorem}[Dense below-boundary regime]\label{thm:dense_below_boundary}
		Let $\alpha+\delta\le1/2$ and $\delta < 0$. Then
		\begin{equation}\label{eq:dense_below_boundary0}
			\lim_{n\to\infty}
			\frac{\log(1-\mathcal R_n(A^\ast(\delta, \alpha)))}{\log n} =\delta.
		\end{equation}
		In fact, there exist absolute constants $c,C\in(0,\infty)$ and $C'(\delta, \alpha) < \infty$ such that, for all $n \ge C'(\alpha, \delta)$,
		\begin{equation}\label{eq:dense_below_boundary}
			c \le \frac{1-\mathcal R_n(A^\ast(\delta, \alpha))}{n^\delta} 
			\le C.
		\end{equation}
	\end{theorem}
	Next, we move on to the sparse sub-regime, i.e., $\alpha + \delta > 1/2$ which implies 
	that $\alpha > 1/2$ since $\delta < 0$. As in part~1 of Theorem~\ref{thm:sparse_riskbnd}, the rate in Theorem~\ref{thm:sparse_below_boundary} below holds for all $\alpha > 1/2$ provided the sparse regime signal strength $A^\ast_{\rm S}(\delta, \alpha)$ is used. Now note from \eqref{def:AdeltaBI} that $A^\ast_{\bullet}(\delta, \alpha) = A^\ast_{\rm S}(\delta, \alpha)$ when $\alpha > 1/2 \ge \alpha + \delta$. But this is 
	also the region where $A^\ast_{\bullet}(\delta, \alpha)$ and $A^\ast(\delta, \alpha)$ are different for $\delta < 0$. Thus Theorem~\ref{thm:sparse_below_boundary} 
	covers the analogue of Theorem~\ref{thm:sparse_riskbnd}, part~2 in this regime.
	\begin{theorem}[Sparse below-boundary regime]\label{thm:sparse_below_boundary}
		Suppose that $\alpha>1/2$ and $\delta\in(-1,0)$. Then
		\begin{equation}\label{eq:sparse_below_boundary0}
			\lim_{n\to\infty}
			\frac{\log(1-\mathcal R_n(A^\ast_{\rm S}(\delta, \alpha)))}{\log n}
			=\frac12-\alpha.
		\end{equation}
		Moreover, there exist absolute constants $c,C\in(0,\infty)$ and $C'(\delta, \alpha) < \infty$ such that, for all $n \ge C'(\alpha, \delta)$,
		\begin{equation}\label{eq:sparse_below_boundary}
			c \le \frac{1-\mathcal R_n(A^\ast_{\rm S}(\delta, \alpha))}{\sqrt{(1+\delta)\alpha}\,n^{1/2-\alpha}\sqrt{\log n}} \le C.
		\end{equation}
	\end{theorem}
	
	\begin{remark}
		Theorems~\ref{thm:dense_below_boundary} and \ref{thm:sparse_below_boundary} reveal a qualitatively different behavior compared to the above-boundary phase studied in 
		Theorems~\ref{thm:dense_above_boundary} and \ref{thm:sparse_riskbnd}. Specifically, the minimax risks approach their limiting value at a much slower polynomial rate, rather than at a stretched or super-exponential speed. This distinction is also reflected in the proof mechanisms for the two regimes, particularly in the lower bounds. While the above-boundary rates emerge from a single free energy-like term dominating the second moment of a suitably truncated statistic (e.g., the likelihood ratio), the below-boundary rates result from the combined contributions of several distinct terms; see \S\ref{subsec:overview} for more details. 	
	\end{remark}

	\subsubsection{Critical regime $(\delta=0)$}\label{subsec:main_on_bnd}
	We finally record a critical result in a fixed-submatrix-size sparse regime. Here 
	$n$ remains bounded and $N$ grows to $\infty$ although our proof can be easily 
	adapted to accommodate the regime where $n$ grows sufficiently slowly as a 
	function of $N$. Since we are working in the extremely sparse limit (i.e., $\alpha = \infty$), the value of the critical threshold $A_{\rm c}^\ast$ corresponding to \eqref{def:Adelta} in 
	this case is given by $\sqrt{\tfrac{4\log N}{n}}$ (note that $\alpha \log n = \log 
	\frac Nn$ when $N = n^{1 + \alpha}$).
	\begin{theorem}[A fixed-size critical limit]\label{thm:criticality}
		For any fixed $n$, 
		\begin{equation}\label{eq:lim_critical}
			\lim_{N\to\infty} \mathcal R(A_{\rm c}^\ast, n, N) = \frac12.
		\end{equation}
	\end{theorem}
	Results of similar flavor have appeared previously in the context of Gaussian 
	sequence models \citep[Theorem~8.1]{ingster2012nonparametric} and linear 
	regression models \citep[Theorem~2.5]{mukherjee2020minimax}. However, 
	\eqref{eq:lim_critical} is different from these results as the limiting value does 
	{\em not} depend on the size of the signal $n$.
	
	\subsection{ Rectangular subtensor extension}\label{subsec:tensor_extension}
	We now collect the extensions of the preceding results to the subtensor 
	detection problem considered in \S\ref{subsec:tensor_summary}. Recall the notations and assumptions in \eqref{def:Theta_tensor}--\eqref{def:AdeltadBI}. We would also need the tensor analogue of \eqref{def:sparse_profile}, namely
	\begin{equation}\label{def:tensor_sparse_profile}
		A^\ast_{d, {\rm S}}(\delta, \alpha) = A^\ast_{d, {\rm S}}(\delta, \alpha; n) \stackrel{{\rm def.}}{=} \sqrt{\frac{2d(1+\delta)\alpha \log n}{n^{d-1}}}.
	\end{equation}
	
	\begin{theorem}[Subtensor extension]\label{thm:tensor_extension}
		Fix $d\ge2$, $\alpha \in (0,\infty)$, and let $(n, N)$ range over any sequence satisfying \eqref{eq:mnregime} (except in part~3).
		\begin{enumerate}[label=\arabic*.]
			\item \emph{Above the detection boundary.} Suppose that $\delta>0$.
			\begin{enumerate}[label=(\alph*)]
				\item If $\alpha + \delta \le 1/d$, then
				\begin{equation}\label{eq:tensor_dense_above}
					\lim_{n \to \infty}
					\frac{\log\mathcal R_{d, n}(A_{d}^\ast(\delta,\alpha))}
					{n^{d \delta}}=-\frac18.
				\end{equation}
				\item For every $\alpha \in (0,\infty)$,
				\begin{equation}\label{eq:tensor_sparse_above}
					\lim_{n \to \infty}
					\frac{\log\mathcal R_{d, n}^{(d)}(A_{d,{\rm S}}^\ast(\delta,\alpha))}
					{\log|\mathcal S_d|}
					=-\frac{\delta^2}{4(1+\delta)}.
				\end{equation}
				In particular, this gives the risk at $A_d^\ast(\delta,\alpha)$ whenever $\alpha + \delta > 1/d$.
				\item If $\alpha \le 1 / d < \alpha + \delta$, then
				\begin{equation}\label{eq:tensor_BI_above}
					\lim_{\bn}
					\frac{\log\mathcal R_{d, n}(A_{d,\bullet}^\ast(\delta, \alpha))}
					{n^{d(\alpha + \delta)}}=-\frac18.
				\end{equation}
			\end{enumerate}
			
			\item \emph{Below the detection boundary.} Suppose that $\delta\in(-1,0)$. Below $c,C\in(0,\infty)$ are absolute constants and $C' = C'(d, \delta, \alpha) < \infty$.
			\begin{enumerate}[label=(\alph*)]
				\item If $\alpha + \delta \le 1/d$, we have for all $n \ge C'$,
				\begin{equation}\label{eq:tensor_dense_below}
					c \le \frac{1-\mathcal R_{d, n}(A_{d}^\ast(\delta,\alpha))}
					{n^{d\delta/2}} \le C.
				\end{equation}
				\item If $\alpha > 1/d$, then for all $n \ge C'$,
				\begin{equation}\label{eq:tensor_sparse_below}
					c \le  \frac{1-\mathcal R_{d, n}(A_{d,{\rm S}}^\ast(\delta,\alpha))}
					{\sqrt{(1+\delta)\alpha}\, n^{1/2 - d\alpha/2}\sqrt{\log n}} \le C.
				\end{equation}
			\end{enumerate}
			
			\item \emph{Fixed-size critical regime.} Let $n$ be fixed. Then, with $A_{d,{\rm c}}^\ast = \sqrt{2d\log N/n^{d-1}}$, we have
			\begin{equation}\label{eq:tensor_critical}
				\mathcal R_d(A_{d,{\rm c}}^\ast, n, N)\longrightarrow\frac12.
			\end{equation}
		\end{enumerate}
	\end{theorem}

	\section{Discussions and Future Directions}\label{sec:discussions}
	\subsection{Connections to existing literature}\label{subsec:connection} The results in this paper are 
	connected both to the classical foundations of statistical efficiency theory and to recent research on sparse signal detection. We elaborate on these connections below and refer the reader to 
	\cite[Section~1.4]{mukherjee2020minimax} for more detailed discussions in a related context.
	
	\subsubsection{Connections to Chernoff exponents and Bahadur efficiency} As mentioned in the introduction, our perspective is closely aligned with the classical theory of asymptotic error exponents in hypothesis testing. For testing a simple null $P$ against a 
	fixed simple alternative $Q$ from $m$ independent observations, the optimal sum of Type~I and Type~II error probabilities is governed, on the exponential scale, by the {\em Chernoff exponent}
	\begin{equation*}
		-   \inf_{\lambda \in \R}  \log \int \left(\frac{dQ}{dP}\right)^{\lambda} dP,
	\end{equation*}
	see, e.g., \cite[Corollary~13.1]{polyanskiy2014lecture}. {\em Bahadur efficiency} provides another large deviation framework for comparing test statistics based on the exponential decay rate of their attained significance levels under the alternative \citep{bahadur1967rates}. 
    Both concepts address the `second order behavior' of the underlying testing problem (see \cite{mukherjee2020minimax}) rather than consistency alone. In this regard, our results can be viewed as providing 
    an analogue of Chernoff exponent and Bahadur efficiency for sparse 
    submatrix (and subtensor) detection.
	
	\vspace{0.04cm}
	
	Nevertheless, the present setting differs from classical theory in crucial respects. Sparse submatrix detection is a simple-versus-composite 
	rather than a simple-versus-simple testing problem. Consequently, the 
	likelihood ratio 
	$L_{\bpi}$ (see \eqref{def:Lpi}) entering the lower bounds resembles a $U$-statistic, i.e., an average of a `large' number of {\em correlated} random variables, as opposed to a product of i.i.d. likelihood ratios 
	underlying classical large deviation formulations. This is not merely a 
	technical obstacle which could be the case, for instance, if the logarithmic decay rate of the minimax risk scaled uniformly with respect to $n$ throughout the above-boundary regime where risks vanish asymptotically. Instead, our results show that the nature of decay of testing errors is highly sensitive to the parameters $(\alpha, \delta)$.

	\subsubsection{Connections to sparse Gaussian mixture models} The sparse Gaussian mixture model is a closely related setting that exhibits 
	a regime-dependent minimax separation rate similar to \eqref{def:Adelta}. In its standard formulation, one observes independent variables $Y_1, 
	\ldots, Y_n$ from one of two distributions:
	\begin{equation*}
		H_0:\;Y_i\sim N(0,1),
		\qquad
		H_1:\;Y_i\sim (1-\varepsilon_n)N(0,1)+\varepsilon_n N(\mu_n,1), \quad 1 \le i \le n.
	\end{equation*}
The detection-boundary for such models has been widely studied in the literature; see, e.g., \cite{donoho2004higher}; \cite{cai2011optimal}; \cite{caiwu2014optimal}; \cite{ariascastro2019sparse}. More recently, \cite{ligo2016detecting} investigated the rate of decay of error probabilities for sparse mixtures in {\em some} regimes. 
	
    However, the sparse Gaussian mixture and sparse submatrix detection are 
    fundamentally different in structure. In the mixture model, the 
	likelihood ratio is still a product over independent observations, reducing 
	the analysis to large deviations for sums of independent random variables. 
	As noted above, such reductions are infeasible for submatrix detection 
	due to the simple-versus-composite nature of the problem.
	
	\subsubsection{Comparisons with \cite{mukherjee2020minimax}}
	A particularly close antecedent to our work is 
	\cite{mukherjee2020minimax}, which studies minimax exponents for sparse testing in the context of high-dimensional linear regression. Under 
	orthogonal designs, this reduces to the Gaussian sequence model. While 
	the authors establish sharp risk exponents in several regimes, their Section~4 explicitly identifies two parameter regions in which the minimax 
	risk remained uncharacterized.

	By contrast, our work focusses on the sparse (square) submatrix model considered in \cite{butuceasubmatrix2013}, where alternatives are parametrized by product sets rather than unstructured subsets of $[N]\times[N]$. For this model, and its extension to (equal-sidelength) tensors, we establish a {\em complete} phase diagram spanning the 
	above-boundary and below-boundary regimes. The qualitative picture of these phase transitions remains identical for any order-$d$ tensor with $d \ge 2$, including a secondary transition around the line 
	$\alpha+\delta=1/d$. Together with the partial characterizations in 
	\cite{mukherjee2020minimax}, these findings demonstrate that $d = 1$ and $d \ge 2$ represent distinct classes of testing problems with respect 
	to minimax risk asymptotics.
	
	From a methodological point of view, \cite{mukherjee2020minimax} employs event truncation (see \S\ref{subsec:overview}) in an effective yet somewhat {\em ad hoc} manner. One main contribution of the present paper is the development of a systematic framework for incorporating truncation engineered specifically to establish {\em target} lower bounds 
	on minimax risks. This framework is quite robust and extends naturally, for 
	instance, to general rectangular subtensor detection and other 
	combinatorial testing problems.
	
	\subsection{Computational considerations}
	The minimax results in this paper are information-theoretic. In 
	scan-dominated regimes, the optimal test requires maximizing $X_S$ over all $S\in\mathcal S$ (see \eqref{def:sumscantest}), which is computationally prohibitive. Computational-statistical tradeoffs for 
	submatrix detection were studied by \cite{ma2015computational}, who 
	linked efficient detection to the planted clique detection problem. 
	Subsequently, \cite{MR4382029} addressed analogous questions for  
	subtensor detection. See also the recent work by \cite{MR4928482} and the references therein for related estimation problems. More 
	broadly, such statistical-computational gaps are known or conjectured in sparse PCA and related planted-structure problems; see, for example, \cite{berthet2013complexity}.
	
	It is therefore a very interesting question to determine the sharp risk 
	asymptotics achievable by randomized polynomial time algorithms. 
	
	\subsection{Other models and further directions}\label{subsec:directions}
	Several natural directions for extension arise from this work. 
	
	\vspace{0.2cm}
	
	\noindent{\em Detection of subtensors with general side lengths.}
	As noted previously, using our methodology one can obtain sharp minimax 
	rates for detecting sparse bounded-aspect-ratio subtensors embedded 
	inside tensors with different scaling exponents across dimensions 
	(see~\eqref{eq:tensor_regime}). This level of generality closely parallels --- though remains slightly more restrictive than --- the settings considered in \cite{butuceasubmatrix2013}; \cite{MR4382029}.
	
	More recently, \cite{knight2026minimax} established non-asymptotic minimax separation rates for submatrix detection allowing arbitrary side 
	lengths for both the embedded and ambient matrices. Extending our sharp 
	risk asymptotics to this fully general is a natural next step which would 
	give rise to a correspondingly rich phase transition picture.

	\vspace{0.15cm}
	
	\noindent{\em Critical window.} An important open problem is the development of a complete critical-window theory --- that is, characterizing the precise scaling of signal strengths for which the limiting minimax risk lies strictly between 0 and 1.
	
	\vspace{0.15cm}
	
	\noindent{\em Non-Gaussian families and universality.} Our reliance on 
	specific properties of the Gaussian distribution is restricted to the tail 
	bounds and change of measure arguments collected in Lemmas~\ref{lem:normal_tailbnd}--\ref{lem:CM}. Analogues of some of 
	these properties continue to hold, in some form, when the underlying 
	noise comes from a general exponential family (see, e.g., Section~4.2 in \cite{butuceasubmatrix2013}). Consequently, an intriguing direction for 
	future research is to investigate the {\em universality classes} of minimax exponents across different noise distributions.

	\vspace{0.15cm}
	
	\noindent{\em Two-sided signals and signals characterized by $\ell_p$-norms.} Testing against two-sided signals poses another natural extension. Based on preliminary calculations, our techniques yield sharp 
	rates in the sparse regime without major modifications. In the dense regime, however, both the minimax separation rate and the optimal test differ. Specifically, the linear sum test is naturally superseded by a 
	chi-square analogue. The main subtlety lies in establishing a matching 
	lower bound. Leveraging our new lower-bound framework, we anticipate 
	that sharp matching lower bounds remain attainable when the signal 
	strength is still governed by its minimum absolute value on its support. But when the signal strength is instead measured by its $\ell_p$-norm 
	deviation from the null, obtaining sharp minimax risk asymptotics would 
	require a delicate analysis.
	
	\section{Preliminaries and lower bound inequalities}\label{sec:prelim}
	In this section we develop the common technical framework used in the proofs. The 
	main ingredients are two general lower bounds on the minimax 
	risk, stated in Propositions~\ref{prop:risk_lowerbnd_ab} and \ref{prop:risk_lowerbnd_bb}, which are respectively adapted to the above- and below-boundary regimes. We conclude the section with some standard 
	properties of hypergeometric and Gaussian distributions that will be used throughout.
	
	\smallskip

	We now present two lower bounds on the minimax risk, each {\em pushing} 
	truncation with respect to a given family of events in a distinct manner. 
	Although stated for the square-matrix problem \eqref{state:test_prob}, they 
	readily extend to other hypothesis testing settings, such as the subtensor 
	problem from \S\ref{subsec:tensor_extension} and more general 
	combinatorial testing problems like those studied in 
	\cite{addarioberry2010combinatorial}.

	\vspace{0.04cm}

	A good way to introduce truncation is through events that depend on a 
	parameter $\btheta \in \Theta(A)$ in addition to the observation matrix $\bX$. To this end, 
	let $\bE  \subset \R^{N \times N} \times \Theta(A) \subset (\R^{N \times 
		N})^2$ (recall \eqref{def:Theta}) denote an event measurable with respect to $\bX$ and 
	$\btheta$, i.e., a Borel subset of $(\R^{N \times N})^2$. Now given a prior probability distribution $\bp$ on $\Theta(A)$, we define the 
	corresponding integrated likelihood ratio truncated with respect to $\bE$ as follows:
	\begin{equation}\label{def:LpiE}
		L_{\bp}^{\bE} = \int_{\Theta(A)} \, \frac{d\P_{\btheta}}{d\P_0} \mathbf 1_{\bE} \, d\bp(\btheta) = \int_{\Theta(A)} \, \exp\left( \langle \btheta, \bX \rangle - \tfrac{\|\btheta\|_{{\rm HS}}^2}{2}\right) \mathbf 1_{\bE} \, d\bp(\btheta),
	\end{equation}
	where $\langle \btheta, \bX \rangle = \sum_{1 \le i, j \le N} X_{ij} \theta_{ij}$ and $\| \btheta \|_{{\rm HS}}^2 = \sum_{1 \le i, j \le N} \theta_{ij}^2$ (see, e.g., 
	\eqref{def:Lpi} and Lemma~\ref{lem:CM} below). For our purpose we will exclusively use the uniform prior 
	$\bpi$ on $\partial \Theta(A)$ (cf.~\eqref{def:partialTheta}). In this case we can 
	think of $\bE$ as a family of $\bX$-measurable events $(E_S)_{S \in \mathcal S}$ (recall \eqref{def:Sfamily} for $\mathcal S$) and $L_{\bpi}^{\bE}$ takes the following simpler form:
	\begin{equation}\label{def:Lpidis}
		L_{\bpi}^{\bE} = \frac{1}{|\mathcal S|} \sum_{S \in \mathcal S} \, \frac{d\P_{\btheta(S)}}{d\P_0} \mathbf 1_{E_S}  = \frac{1}{|\mathcal S|} \sum_{S \in \mathcal S} \exp\left(A X_S - \tfrac{A^2n^2}{2}\right)\mathbf 1_{E_S}
	\end{equation}
	with $X_S =\sum_{(i, j)} X_{ij}$ as in \eqref{def:XS} and
	\begin{equation}\label{def:thetaS}
		\btheta(S) \stackrel{\rm def.}{=} (A\mathbf 1_{(i, j) \in S})_{1 \le i, j \le N}.
	\end{equation}
	Notice that $L_{\bp}^{\bE}$, and hence $L_{\bpi}^{\bE}$, is measurable with respect to $\bX$. In the 
	{\em untruncated} case $\bE = \R^{N \times N} \times \Theta(A)$, $L_{\bp}^{\bE}$ is simply the integrated likelihood ratio $L_{\bp} = \frac{d 
		\P_{\bp}}{d\P_0}$ where $\P_{\bp} = \int_{\Theta(A)} \P_{\btheta} \, d\bp(\btheta)$ denotes the corresponding mixture measure. Thus
	\begin{equation}\label{eq:Lpiform}
		\begin{split}
			& L_{\bp} = \frac{d \P_{\bp}}{d\P_0} = \int_{\Theta(A)}\exp\left( \langle \btheta, \bX \rangle - \tfrac{\|\btheta\|_{{\rm HS}}^2}{2}\right)  \, d\bp(\btheta), \:\: \mbox{and}\\
			& L_{\bpi} = \frac{d \P_{\bpi}}{d\P_0} = \frac{1}{|\mathcal S|} \sum_{S \in \mathcal S} \exp\left(A X_S - \tfrac{A^2n^2}{2}\right)
		\end{split}
	\end{equation}
	(cf.~\eqref{def:LpiE}--\eqref{def:Lpidis}). We can now state the inequalities.
	\begin{proposition}[A Paley--Zygmund type inequality for minimax risk]\label{prop:risk_lowerbnd_ab}
		For any event $\bE$ measurable with respect to $\bX$ and $\btheta$, prior distribution $\bp$ and 
		signal strength $A \in (0, \infty)$, we have the lower bound
		\begin{equation}\label{eq:risk_lowerbnd_ab}
			\mathcal R(A, n, N) \ge \min\left(\tfrac 1 2 \, \E_{0}( L_{\bp}^{\bE}),\,  \frac{(\E_0( L_{\bp}^{\bE}))^2}{4 \, \E_{0}\left( (  L_{\bp}^{\bE} )^2 \right)}\right).
		\end{equation}
	\end{proposition}
	\begin{proof}
		Since the maximum risk is bounded below by the Bayes risk under the 
		prior $\bp$ which is itself the risk for testing $\P_0$ against $\P_{\bp}$ 
		with corresponding likelihood ratio $L_{\bp}$ (see \eqref{eq:Lpiform} above), it follows from the Neyman-Pearson lemma that
		\begin{equation}\label{eq:NPearson0}
			\mathcal R(A, n, N) \ge \P_0(L_{\bp} > 1) + \P_{\bp}(L_{\bp} \le 1) \ge \max(\P_0(L_{\bp} > 1), \P_{\bp}(L_{\bp} \le 1)).
		\end{equation}
		From the Cauchy-Schwarz inequality, we obtain
		\begin{equation}\label{eq:CS1}
			\P_0( L_{\bp} > 1) \ge \frac{(\E_0(L_{\bp}^{\bE}  \mathbf 1_{ \{L_{\bp} > 1 \}}))^2}{\E_0((L_{\bp}^{\bE})^2)}.
		\end{equation}
		Now note that
		\begin{equation*}
			\begin{split}
				\E_0(L_{\bp}^{\bE} \mathbf 1_{\{L_{\bp} > 1\}}) &= \E_{0}( L_{\bp}^{\bE} ) - \E_0(L_{\bp}^{\bE} \mathbf 1_{\{L_{\bp} \le 1\}}) \stackrel{(L_{\bp}^{\bE} \le L_{\bp}) }{\ge} \E_{0}( L_{\bp}^{\bE} ) - \E_0(L_{\bp} \mathbf 1_{\{L_{\bp} \le 1\}}) \\
				& \stackrel{\eqref{eq:Lpiform}}{=} \E_{0}( L_{\bp}^{\bE} ) - \P_{\bp}(L_{\bp} \le 1).
			\end{split}
		\end{equation*}
		Therefore, 
		\begin{equation*}
			\max\left( \E_0(L_{\bp}^{\bE} \mathbf 1_{\{L_{\bp} > 1\}}),  \P_{\bp}(L_{\bp} \le 1)\right) \ge \frac12\E_{0}( L_{\bp}^{\bE} ).
		\end{equation*}
		Together with \eqref{eq:NPearson0} and \eqref{eq:CS1}, this implies \eqref{eq:risk_lowerbnd_ab}.
	\end{proof}
	The second inequality is a generalization of the classical lower bound \eqref{eq:CS}.
	\begin{proposition}\label{prop:risk_lowerbnd_bb}
		Under the same setup as in Proposition~\ref{prop:risk_lowerbnd_ab}, we have
		\begin{equation}\label{eq:risk_lowerbnd_bb}
			\mathcal R(A, n, N) \ge \E_0(L_{\bp}^{\bE}) - \frac12\sqrt{\var_0(L_{\bp}^{\bE})}.
		\end{equation}
	\end{proposition}
	\begin{proof}
		Starting as in \eqref{eq:NPearson0}, we can write
		\begin{equation}\label{eq:NPearson}
			\begin{split}
				\mathcal R(A, n, N) &\ge \P_0(L_{\bp} > 1) + \P_{\bp}(L_{\bp} \le 1) \stackrel{\eqref{eq:Lpiform}}{=} \P_0(L_{\bp} > 1) + \E_0(L_{\bp}\mathbf1_{\{L_{\bp} \le 1\}}) \\
				&= 1 -  \E_0(1 - L_{\bp})^+
			\end{split}
		\end{equation}
		where $x^+ \stackrel{{\rm def.}}{=} \max(x, 0)$ for any $x \in \R$. Now decomposing $L_{\bp}$ (resp. $1$) as $(L_{\bp} - L_{\bp}^{\bE}) + L_{\bp}^{\bE}$ (resp. $(1 - \E_0(L_{\bp}^{\bE})) + \E_0(L_{\bp}^{\bE})$), which are both non-negative, and using the elementary observation that the function $x \mapsto x^+$ is both increasing and subadditive, we obtain
		\begin{equation*}
			\begin{split}
				\mathcal R(A, n, N) \ge 1 -  (1 - \E_0(L_{\bp}^{\bE})) -\E_0\big(\E_0(L_{\bp}^{\bE}) - L_{\bp}^{\bE}\big)^+ = \E_0(L_{\bp}^{\bE}) - \frac12 \E_0\big|\E_0(L_{\bp}^{\bE}) - L_{\bp}^{\bE}|.
			\end{split}    
		\end{equation*}
		From this \eqref{eq:risk_lowerbnd_bb} follows immediately upon noting that 
		$\E_0\big|\E_0(L_{\bp}^{\bE}) - L_{\bp}^{\bE}| \le \sqrt{\var_0(L_{\bp}^{\bE})}$ 
		by the Cauchy-Schwarz inequality.
	\end{proof}
	\begin{remark}[Regarding the choice of $\bE$]\label{rm:bE}
		Both \eqref{eq:risk_lowerbnd_ab} and \eqref{eq:risk_lowerbnd_bb} are designed to give a lower bound around the quantity $\E_0( L_{\bp}^{\bE})$. In view of the definition \eqref{def:LpiE}, the latter can be expressed as
		\begin{equation}\label{rewrit:ELpiE}
			\E_0(L_{\bp}^{\bE}) = \int_{\Theta(A)} \P_{\btheta} (\bE(\btheta) ) \, d\bp(\btheta)
		\end{equation}
		where the event $\bE(\btheta) \stackrel{{\rm def.}}{=} \{ \mathbf{X} \in   \R^{N \times N} :  (\mathsf X, \btheta) \in \bE \}$ is measurable with respect to $\bX$ (cf.~the events $(E_S)_{S \in \mathcal S}$ in \eqref{def:Lpidis}). Thus we would like to choose the events $\bE(\btheta)$ so that (a) $\P_{\btheta} (\bE(\btheta) )$ is close to the target lower bound for all $\btheta \in {\rm supp}(\bp)$ and (b) the second moment of 
		$L_{\bp}^{\bE}$ is `well-behaved'.  When the target bound is small, like in the above-boundary regime considered by 
		Theorems~\ref{thm:dense_above_boundary}--\ref{thm:sparse_riskbnd}, Proposition~\ref{prop:risk_lowerbnd_ab} becomes useful to derive sharp 
		lower bounds. On the other hand when it is bounded away from $0$, as in the below-boundary and critical regimes covered by 
		Theorems~\ref{thm:dense_below_boundary}--\ref{thm:criticality}, we use 
		Proposition~\ref{prop:risk_lowerbnd_bb}. See \S\ref{subsec:overview} for a more 
		detailed discussion.
	\end{remark}

	Next we state some estimates on the probability mass function of the  Hypergeometric distribution ${\rm HG}(N, n, n)$. Note that the row and the column overlap size between a uniformly chosen $n \times n$ submatrix and a fixed $n \times n$ submatrix are each distributed as independent ${\rm HG}(N, n, n)$ variables. Control over these overlap statistics is crucial for the delicate second moment computations underlying our main lower bounds (via Propositions~\ref{prop:risk_lowerbnd_ab} and \ref{prop:risk_lowerbnd_bb}). In 
	the sequel, we interpret $0 \log \infty = 0$.
	\begin{lemma}[Hypergeometric probability estimates]\label{lem:factorial}
		For any non-negative integer $k$ and positive integers $n$ and $N$ satisfying $k \le n \le \tfrac N 3$, we have
		\begin{equation}\label{eq:factorial_dense}
			\tfrac{c}{\sqrt{k \vee 1}\frac {(n-k)\vee1} n} \le \frac{\binom{n}{k}\binom{N - n}{n - k}}{\binom{N}{n}}\exp( -g(k))  \le \tfrac{C}{\sqrt{k \vee 1}\frac {(n-k)\vee1} n}  
		\end{equation}
		where, for any $x \ge 0$,
		\begin{equation}\label{def:g}
			g(x) = g_{n, N}(x) \stackrel{{\rm def.}}{=} x \log \tfrac{n^2}{Nx} - 2Nh(\tfrac{N-n}N) + 2n h(\tfrac {n - x}n) + Nh(\tfrac{N - 2n + x}N)
		\end{equation}
		with $h(x) \stackrel{{\rm def.}}{=} x \log \tfrac 1x$. Furthermore, 
		we have the simplified bound
		\begin{equation}\label{eq:factorial}
			\frac{\binom{n}{k}\binom{N - n}{n - k}}{\binom{N}{n}} \le 
			\exp(k\log \tfrac {en^2}{Nk}).
		\end{equation}
	\end{lemma}
	\begin{proof}
		From Stirling's formula, we have for any positive integers $m > k$,
		\begin{equation}\label{eq:stirlingmk}
				\binom{m}{k} \asymp \frac{1}{\sqrt{k(1 - \frac km)}} \exp\big(k \log \tfrac mk + (m - k) \log \tfrac m{m - k}\big)
		\end{equation}
		where, in this proof, we use `$\asymp$' to indicate that the ratio of the 
		expressions on its both sides is bounded between $\frac 1 C$ and $C$ for 
		some absolute constant $C \in (1, \infty)$. 
		Now consider non-negative integers $k, n$ and $N$ satisfying $k < n < \frac N2$. Using \eqref{eq:stirlingmk}, we can write
		\begin{equation}\label{eq:Nnk}
			\binom{N - n}{n-k} \asymp \frac{1}{\sqrt{(n - k)(1 - \frac{n - k}{N - n})}}  \exp\big((n-k) \log \tfrac {N-n}{n-k} + (N - 2n + k) \log \tfrac{N - n}{N - 2n + k}\big).
		\end{equation}
		We can rewrite the first term inside the exponential as follows:
		\begin{equation}\label{eq:N2nk}
			(n-k) \log \tfrac{N-n}{n-k} = (n - k) \log \tfrac N n - (n - k) \log \tfrac N{N-n} + (n - k) \log \tfrac n{n - k}.
		\end{equation}
		As to the other term, we can write
		\begin{equation}\label{eq:decomp}
			(N - 2n + k)\log \tfrac {N - n}{N - 2n + k} = (N - 2n + k) \log \tfrac{N}{N - 2n + 
				k} - (N - 2n + k)\log \tfrac{N}{N - n}.
		\end{equation}
		The expression of $\binom{N - n}{n-k}$ resulting from \eqref{eq:Nnk} and 
		\eqref{eq:N2nk} combined with the expressions of $\binom{n}{k}$ and $\binom{N}{n}$ 
		obtained from \eqref{eq:stirlingmk} give us \eqref{eq:factorial_dense} in the 
		regime $1 \le k < n \le \tfrac N3$. The terminal cases $k = 0$ and $k = n$ can be 
		verified similarly upon noting that $\binom{n}{k} = 1$ in both cases and $\binom{N - n}{n - k} = 1$ in the latter case.
		
		\smallskip
		
		The simplified bound \eqref{eq:factorial} is more elementary. First notice that,
		\begin{equation*}
			\binom{n}{k} = \frac{n(n -1) \cdots (n - k + 1)}{k !} \le \frac{n^k}{k!} \le \left(\frac{n}{k}\right)^k \frac{k^k}{k !} \le \left(\frac{e n}{k}\right)^k.
		\end{equation*}
		On the other hand, using $k \le n$ and $n \le N$ in the first and third steps 
		below respectively, we obtain
		\begin{equation*}
			\begin{split}
			\frac{\binom{N - n}{n - k}}{\binom{N}{n}} \le \frac{\binom{N - k}{n - k}}{\binom{N}{n}} = \frac{n (n -1) \cdots (n - k + 1)}{N (N -1) \cdots (N - k + 1)} \le \left(\frac{n}{N}\right)^k.
		\end{split}
		\end{equation*}
		Together, these two estimates give us \eqref{eq:factorial_dense}.
	\end{proof}
	In some situations, it is enough to work with the following simplified bound.
	\begin{corollary}\label{cor:factorial} 
		For any non-negative integer $k$ and positive integers $n$ and $N$ satisfying $k 
		\le n \le \tfrac N 3$, we have
		\begin{equation}\label{eq:factorial1}
			\frac{\binom{n}{k}\binom{N - n}{n - k}}{\binom{N}{n}} \le \exp(C n - k\log \tfrac{N}{n}).\end{equation}
	\end{corollary}
	\begin{proof}
		\eqref{eq:factorial1} follows immediately from \eqref{eq:factorial_dense} 
		upon noting that $0 \le \inf_{x \in [0, 1]} h(x) \le \sup_{x \in [0, 1]} h(x) < \infty$ 
		and $h(x) \le C(1 - x)$ for $x \in [\frac 12, 1]$. We skip the verification of 
		these two properties which are standard.
	\end{proof}
	The following convex stochastic order relation between Hypergeometric 
     and Binomial variables is often useful. See, 
     e.g.,~\citep{aldous2006exchangeability} for a proof. 
	\begin{lemma}\label{lem:conv_order}
	Let $V \sim {\rm HG}(N, n, n)$ and $W \sim {\rm Bin}(n, \frac n N)$. Then for all convex functions $\phi : \R \to \R$, we have
	\begin{equation}\label{eq:stoch_order}
		\E_{{\rm HG}} (\phi(V)) \le \E_{{\rm Bin}} (\phi(W))
	\end{equation}
	where $\E_{{\rm HG}}$ and $\E_{{\rm Bin}}$ are expectations with respect to the laws of $V$ and $W$ respectively.
	\end{lemma}

	\smallskip
	
	To conclude the section, we record the properties of normal variables that we use in our proofs. The bound below is classical; see, e.g., \cite[Proposition~2.1.2]{MR3837109}. 
	\begin{lemma}[Standard normal tail bounds]\label{lem:normal_tailbnd}
		Letting $\Phi(\cdot)$ denote the standard normal cumulative distribution 
		function (CDF), we have
		\begin{equation}\label{eq:normal_tailbnd}
			\frac{1}{\sqrt{2\pi}}\frac{x}{x^2 + 1}e^{-\frac{x^2}{2}} \le 1 - \Phi(x) = \Phi(-x) \le \tfrac 1{2\vee \sqrt{2\pi}x} e^{-\frac{x^2}{2}}
		\end{equation}
		for all $x \ge 0$.
	\end{lemma}
	\begin{lemma}[Joint tail probability estimates]\label{lem:bvn}
		Let $S, S' \subset [N] \times [N]$ be such that $|S| = |S'| = n^2$. Then for any $a \ge 0$, we have 
		\begin{equation}\label{eq:bvn}
			\P_0\big(\tfrac{X_S}n \ge a,  \tfrac{X_{S'}}n  \ge a\big) =  \P_0\big(\tfrac{X_S}n \le -a,  \tfrac{X_{S'}}n  \le -a\big) \le 
			\exp\left( - \frac{a^2}{1 + \frac{|S \cap S'|}{n^2}}\right)
		\end{equation}
		where $X_S, X_{S'}$ are as in \eqref{def:XS}.
	\end{lemma}
\begin{proof}
Note that
\begin{equation*}
\P_0\big(\tfrac{X_S}n \ge a,  \tfrac{X_{S'}}n  \ge a\big) \le \P_0\big(\tfrac{X_S}n + \tfrac{X_{S'}}n  \ge 2a \big).
\end{equation*}
But $\tfrac{X_S}n + \tfrac{X_{S'}}n$ is distributed as a centered normal random 
variable with variance $2 + 2 \frac{|S \cap S'|}{n^2}$ under $\P_0$. Consequently, by the previous lemma, we obtained
\begin{equation*}
\P_0\big(\tfrac{X_S}n \ge a,  \tfrac{X_{S'}}n  \ge a\big) \le \exp\left( - \tfrac{a^2}{1 + \frac{|S \cap S'|}{n^2}}\right)
\end{equation*}
which is the desired tail bound in \eqref{eq:bvn}.
\end{proof}
	The final result in the list is the `change of mean' formula for normal distribution. In the 
	form presented below, it is a direct consequence of the definition of Gaussian density.
	\begin{lemma}[Change of mean formula]\label{lem:CM}
		Let $\P$ denote the law of $n$ independent standard normal variables. Also for 
		any $\btheta = (\theta_1, \ldots, \theta_m) \in \R^m$, let $\P_{\btheta}$ denote the law of $n$ independent normal variables with means $\theta_1, \ldots, \theta_m$ and common 
		variance $1$. Then for ($\P$ almost) every $\mathbf y = (y_1, \ldots, y_m) \in \R^m$, we have $\frac{d\P_{\btheta}}{d\P}(\mathbf y) = \exp( \sum_{1 \le i \le m} \theta_i y_i - \frac12\sum_{1 \le i \le m} \theta_i^2)$.
	\end{lemma}
	
	\section{Proofs in the above-boundary regime}\label{sec:dense_above}
	In this section, we will prove our results for the square-matrix model in the 
	above-boundary regime $\delta > 0$, i.e., Theorems~\ref{thm:dense_above_boundary} and \ref{thm:sparse_riskbnd}. 
	Proofs of the lower bounds rely on Proposition~\ref{prop:risk_lowerbnd_ab} 
	and manifest a free-energy principle whereby the second moment 
	appearing inside \eqref{eq:risk_lowerbnd_ab} is dominated by a single term; 
	see~\eqref{eq:free_energy1} and \eqref{eq:free_energy2}. Throughout this section, we use $o(1)$ to denote any function of $n$ 
	(depending possibly on $\alpha$ and $\delta$) that tends to $0$ as $n \to \infty$ with $\alpha$ and $\delta$ fixed.
	\subsection{Proof of Theorem~\ref{thm:dense_above_boundary}}\label{subsec:dense_above}
	Here we will give the proof of Theorem~\ref{thm:dense_above_boundary} which is naturally split into the proofs of the corresponding upper and lower bound respectively. 
	
	\medskip
	
	\noindent{{\em Proof of the upper bound.}} For the upper bound, consider the sum test (revisit our discussion in the first paragraph of \S\ref{subsec:overview}) given by
	\begin{equation}\label{def:sumtest_above}
		T = T_{{\rm sum}}^{N \frac{n^{\delta}}2} = \mathbf{1}_{\big\{X_{[N]\times[N]} > N \frac{n^{\delta}}2\big\}}.    
	\end{equation}
	Since $\frac{X_{[N]\times[N]}}{N} \sim N(0, 1)$ under $\P_0$, we have
	\begin{equation}\label{eq:sumdense_upperP0}
		\P_0(T = 1) = 1 - \Phi(\tfrac{n^{\delta}}2) \stackrel{\eqref{eq:normal_tailbnd}}{\le} e^{-\frac{n^{2\delta}}8}.
	\end{equation}
	On the other hand, $\tfrac{X_{[N]\times[N]}}{N} \sim N(\mu_{\btheta}, 1)$ 
	under $\P_{\btheta}$ with $\mu_{\btheta} \ge \tfrac{n^2A}{N}$ and 
	\begin{equation}\label{def:A_dense_above}
		A = A^\ast(\delta, \alpha) = n^{-(1 - \alpha - \delta)}
	\end{equation}
	for any $\btheta \in \Theta(A)$. Therefore
	\begin{equation}\label{eq:sumdense_upperP1}
		\P_{\btheta}(T = 0) \le \Phi\big(\tfrac{n^{\delta}}2 - \tfrac{n^2A}{N}\big) \stackrel{\eqref{eq:mnregime}}{\le} \Phi(-\tfrac{n^{\delta}}2) \stackrel{\eqref{eq:normal_tailbnd}}{\le} e^{-\frac{n^{2\delta}}8}
	\end{equation}
	Together \eqref{eq:sumdense_upperP0} and \eqref{eq:sumdense_upperP1} imply (cf.~\eqref{eqn:bayes_risk})
	\begin{equation*}
		\limsup_{n\rightarrow \infty}\frac{\log{\mathrm{Risk}(T,A,n)}}{n^{2\delta}} \le  -\frac18
	\end{equation*}
	which is the upper bound in Theorem~\ref{thm:dense_above_boundary}.
	
	\medskip
	
	\noindent{{\em Proof of the lower bound.}} We now give the proof for the lower bound 
	which is more involved. Taking a cue from \eqref{eq:sumdense_upperP1} --- namely, 
	that the event $\{T = 0\}$ has the desired probability under $\P_{\btheta}$ for any 
	$\btheta \in \partial \Theta(A)$ --- and Remark~\ref{rm:bE}, we will use 
	Proposition~\ref{prop:risk_lowerbnd_ab} with the choices $\bp = \bpi$ and $$E_S = \{T = 0\}$$ for all $S \in \mathcal S$ (see \eqref{def:Lpidis} and the 
	line before it). With these choices we have, first of all,
	\begin{equation}\label{eq:expect_low}
		\begin{split}
			\E_{0}( L_{\bpi}^{\bE}) &\stackrel{\eqref{rewrit:ELpiE}, \,\eqref{def:Lpidis}}{=} \frac{1}{|\mathcal S|}\sum_{S \in \mathcal S} \P_{\btheta(S)}(T = 0) \,  \stackrel{\eqref{def:sumtest_above}, \,\eqref{def:thetaS}}{=} \, \P_0\big( \tfrac{X_{[N]\times[N]}}{N} \le \tfrac{n^{\delta}}2 -  \tfrac{n^2A}{N}\big) \\
			&\stackrel{\eqref{eq:mnregime}, \, \eqref{def:A_dense_above}}{=} \Phi(-\tfrac{(1 + o(1)) \, n^{\delta}}2) \stackrel{\eqref{eq:normal_tailbnd}}{\ge} e^{-(1 + o(1))\frac{n^{2\delta}}8}
		\end{split}
	\end{equation}
	(recall $o(1)$ from \S\ref{subsec:notation}). We will now evaluate the second moment term 
	in \eqref{eq:risk_lowerbnd_ab} from Proposition~\ref{prop:risk_lowerbnd_ab}. In particular, we would like to show that \begin{equation}\label{eq:truncatedLpisq}
		\E_{0}\left( (  L_{\bpi}^{\bE} )^2 \right) = 	\E_0(L_{\bpi}^2 \mathbf 1_{\{T = 0\}}) \le \Phi(-\tfrac{3n^{\delta}}2) \exp((1 + o(1)) n^{2\delta}) \stackrel{\eqref{eq:normal_tailbnd}}{\le} e^{-(1 + o(1))\frac{n^{2\delta}}8}.
	\end{equation}
	Together \eqref{eq:expect_low} and \eqref{eq:truncatedLpisq} would imply, in view of \eqref{eq:risk_lowerbnd_ab},
	\begin{equation*}
		\liminf_{n\rightarrow \infty}\frac{\log{\mathcal R_n(A)}}{n^{2\delta}} \ge  -\frac18
	\end{equation*}
	which is the required lower bound in Theorem~\ref{thm:dense_above_boundary}. Rest of the proof is devoted to verifying 
	\eqref{eq:truncatedLpisq}.

	Let us start by expanding the second moment (recall that $A = n^{-(1 - \alpha - \delta)}$).
	\begin{equation}\label{eq:LpiT0}
		\begin{split}
			\E_0(L_{\bpi}^2 \mathbf 1_{\{T = 0\}}) &\stackrel{\eqref{eq:Lpiform}}{=} \frac 1{|\mathcal S|^2} \sum_{S, S' \in \mathcal S} \E_0\big(e^{A(X_S + X_{S'}) - A^2n^2} \, \mathbf 1_{\{T = 0\}} \big)\\
			&\stackrel{{\rm Lem}~\ref{lem:CM}, \, \eqref{def:sumtest_above}}{=} \frac 1{|\mathcal S|^2} \sum_{S, S' \in \mathcal S} e^{A^2|S \cap S'|} \, \P_0 \big( X_{[N] \times [N]} \le  N \tfrac{n^\delta}2 - 2An^2 \big)\\ 
			&\stackrel{\eqref{def:A_dense_above}}{=} \Phi(-\tfrac{3n^{\delta}}2) \frac 1{|\mathcal S|^2} \sum_{S, S' \in \mathcal S} e^{A^2|S \cap S'|} = \Phi(-\tfrac{3n^{\delta}}2) \sum_{0 \le k, \ell \le n} \tfrac{N_{k, \ell}}{|\mathcal S|} e^{A^2k \ell}
		\end{split}
	\end{equation}
	where 
	\begin{equation}\label{def:Nkl}
		N_{k, \ell} 
		\stackrel{{\rm def.}}{=} {n \choose k} {N - n \choose n - k} 
		{n \choose \ell}{N - n \choose n - \ell}
	\end{equation}
	is the number of submatrices having $k$ common rows and $\ell$ common 
	columns with a given submatrix. We can bound the summation in the final step as
	\begin{equation}\label{eq:free_energy1}
		\sum_{0 \le k, \ell \le n} \tfrac{N_{k, \ell}}{|\mathcal S|} e^{A^2k \ell} \le Cn^2 \max_{0 \le k, \ell \le n} \tfrac{N_{k, \ell}}{|\mathcal S|} e^{A^2k \ell} = \exp(o(1) n^{2\delta}) \exp \big( - \min_{0 \le k, \ell \le n} ( -A^2 k\ell - \log \tfrac{N_{k, \ell}}{|\mathcal S|} ) \big).
	\end{equation}
	Now using the upper bound part of \eqref{eq:factorial_dense}, we can write
	\begin{equation*}
		\tfrac{N_{k, \ell}}{|\mathcal S|} e^{A^2k \ell} \le C n \exp( g(k) + g(\ell) + A^2 k\ell)
	\end{equation*}
	for $n \ge C(\alpha)$ with $g(x)$ as defined in \eqref{def:g}. Let $(x^\ast, y^\ast)$ denote any point of maximum of the function $g(x) + g(y) + A^2 xy$ on $[0, n]^2$. Then
	\begin{equation*}
		\max_{0 \le k, \ell \le n} \tfrac{N_{k, \ell}}{|\mathcal S|} e^{A^2k \ell} \le C n \exp( g(x^\ast) + g(y^\ast) + A^2 x^\ast y^\ast).
	\end{equation*}
	On the other hand, the lower bound part of \eqref{eq:factorial_dense} tells us that 
	\begin{equation*}
		\frac{c}{n}\exp(g( \lfloor x^\ast \rfloor )  + g( \lfloor y^\ast \rfloor )  )	 \le  \tfrac{N_{\lfloor x^\ast \rfloor, \lfloor y^\ast \rfloor}}{|\mathcal S|}  \le 1.
	\end{equation*}
	Combining these observations with \eqref{eq:free_energy1}, we get
	\begin{equation*}
		\sum_{0 \le k, \ell \le n} \tfrac{N_{k, \ell}}{|\mathcal S|} e^{A^2k \ell} \le \exp(o(1) n^{2\delta}) \, \exp( A^2 x^\ast y^\ast) \, \exp(g(x^\ast) -  g(\lfloor x^\ast \rfloor)) \exp(g(y^\ast) -  g(\lfloor y^\ast \rfloor)).
	\end{equation*}
	It is straightforward to check that $|g(x) - g(\lfloor x \rfloor)| \le \log  \frac{Cn}{(n - x - 1)^+}$ for all $x \in [0, n]$ (with the convention $\frac{Cn}{0} = \infty$) when $n \ge C(\alpha)$. Consequently,
	\begin{equation}\label{eq:free_energy1*}
		\sum_{0 \le k, \ell \le n} \tfrac{N_{k, \ell}}{|\mathcal S|} e^{A^2k \ell} \le \exp(o(1) n^{2\delta}) \, \tfrac{n}{(n - x^\ast - 1)^+}\tfrac{Cn}{(n - y^\ast - 1)^+} \, \exp( A^2 x^\ast y^\ast).
	\end{equation}	
	In order to estimate $(x^\ast, y^\ast)$, recall that it is a point of maximum for the 
	function $g(x) + g(y) + A^2 xy$ or, equivalently, the function $$f(x, y) \stackrel{{\rm def.}}{=} g(x) + g(y) + A^2 xy - 4N h(\tfrac{N - n}{N})$$ on $[0, n]^2$. From the 
	definition of $g(\cdot)$ in Lemma~\ref{lem:factorial}, we then have
	\begin{equation*}
		f(x, y) = x\log\tfrac{n^2}{Nx} + y\log\tfrac{n^2}{Ny} + 2n(h(\tfrac{n - x}n) + h(\tfrac{n - y}n)) + N (h(\tfrac{N - 2n + x}N) + h(\tfrac{N - 2n + y}N)) + A^2xy. 
	\end{equation*}
	$f$ is clearly continuous and its partial derivatives on $(0, n]^2$ are given by
	\begin{equation*}
		\frac{\partial f}{\partial x} =  \log\tfrac{(n-x)^2}{x(N - 2n + x)} + A^2 y, \quad \frac{\partial f}{\partial y} = \log\tfrac{(n-y)^2}{x(N - 2n + y)} + A^2x.
	\end{equation*}
	Since $\alpha + \delta \le \frac 12$, we have $A^2 = n^{-2(1 - \alpha - \delta)}\le \frac1n$ and hence $\max(A^2x, A^2y) \le 1$. From these two observations, we 
	obtain that $(x^\ast, y^\ast)$ satisfies
	\begin{equation*}
		\tfrac{n^2}N \le x^\ast  \le \tfrac{n^2}N(1 + C A^2y^\ast) \mbox{ and } \tfrac{n^2}N \le y^\ast  \le \tfrac{n^2}N(1 + C A^2x^\ast).
	\end{equation*}
	In particular, this implies that both $x^\ast$ and $y^\ast$ are bounded above by 
	$C\tfrac{n^2}N$. Re-plugging these upper bounds into the previous display and 
	recalling that $A^2 \le 1/n$, we get the improved estimate 
	\begin{equation*}
		(x^\ast, y^\ast) \in [\tfrac {n^2}N, \tfrac {n^2}N(1 + C \tfrac n N)]^2.
	\end{equation*}
	As $\frac n N = o(1)$, \eqref{eq:free_energy1*} now gives us
	\begin{equation*}
		\sum_{0 \le k, \ell \le n} \tfrac{N_{k, \ell}}{|\mathcal S|} e^{A^2k \ell} \le \exp(o(1) n^{2\delta}) \exp( (1 + o(1)) \tfrac{A^2n^4}{N^2} ) = \exp( (1 + o(1)) n^{2\delta} ).
	\end{equation*}
	Plugging this bound into \eqref{eq:LpiT0} we obtain \eqref{eq:truncatedLpisq}.
	\qed
	
	\subsection{Proof of Theorem~\ref{thm:sparse_riskbnd}}\label{subsec:sparse}
	This subsection is {\em primarily} devoted to the proof of Theorem~\ref{thm:sparse_riskbnd}, part~1. In the end we briefly explain how the same proof, with a different 
	choice of the parameter $\tau^\ast$ (see below), also leads to part~2. 
	
	\smallskip
	
	We will frequently use the following asymptotic expression for $|\mathcal S| = {\binom{N}
		{n}}^2$ which is valid for {\em all} $\alpha > 0$ (recall from \eqref{eq:mnregime} that $N = \lfloor n^{1 + \alpha} \rfloor$) and follows from the Stirling's formula as in \eqref{eq:stirlingmk}:
	\begin{equation}\label{eq:Scardinality}
		\log |\mathcal S| = (1 + o(1)) \, 2\alpha n \log n.
	\end{equation}
	Clearly, with $A = A_{\rm S}^\ast(\delta, \alpha) = \sqrt{\frac{4(1+\delta)\alpha \log n}{n}}$, we have
	\begin{equation}\label{eq:mathcalSA}
		\log |\mathcal S| = 
		(1 + o(1)) \tfrac{A^2}{2(1 + \delta)} \, n^2.
	\end{equation}
	Our analysis in the proof of Theorem~\ref{thm:sparse_riskbnd} (part~1) is valid for {\em any} $\alpha, \delta > 0$ and $A$ as above.

	\medskip
	
	\noindent{{\em Proof of the upper bound.}} Let us start with the upper bound and 
	consider the scan test (cf.~\eqref{def:sumscantest}) with a carefully chosen threshold, namely
	\begin{equation}\label{def:Tscan}
		T = T_{{\rm scan}}^{\sqrt{2\tau^*\log(|\mathcal{S}|)}} = \mathbf{1}_{\big\{\max\limits_{S\in 
				\mathcal{S}}\frac{X_{S}}{n}>\sqrt{2\tau^*\log(|\mathcal{S}|)}\big\}}\,\, 
		\mbox{where }  \tau^\ast = \frac{(2 + \delta)^2}{4(1 + \delta)}.
	\end{equation}
	Note that $\tau^\ast - 1 = \frac{\delta^2}{4(1 + \delta)} (> 0)$; cf.~\eqref{eq:sparse_riskbnd}.
	Since each $\frac{X_S}{n} \sim N(0, 1)$ under $\P_0$, we can write, via a simple union bound, 
	\begin{equation*}
		\P_0(T = 1) \le |\mathcal S| \, \big(1 - \Phi(\sqrt{2\tau^*\log(|\mathcal{S}|)})\big).
	\end{equation*}
	Plugging the upper tail bound from \eqref{eq:normal_tailbnd} into the right-hand side above, 
	we obtain
	\begin{equation}\label{eq:scan_upperP0}
		\P_0(T = 1) \le 
		e^{-(\tau^\ast - 1) \log |\mathcal S|}.
	\end{equation}
	On the other hand, for any $\btheta \in \Theta(A)$, we have
	\begin{equation*}
		\P_{\btheta}(T = 0) \le  \P_{\btheta}\left(\tfrac{X_{S(\btheta)}}{n}  \le \sqrt{2\tau^*\log(|\mathcal{S}|)} \right)
	\end{equation*}
	where $S(\btheta) \in \mathcal S$ is such that $\theta_{ij} \ge A$ for all $(i, j) 
	\in S(\btheta)$ and $= 0$ otherwise. Since $\frac{X_{S(\btheta)}}{n} \sim N(\mu_{\btheta}, 1)$ under $\P_{\btheta}$ with $\mu_{\btheta} \ge An$, we get
	\begin{equation}\label{eq:scan_upperP2}
		\P_{\btheta}(T = 0) \le   \P_{\btheta}\left(\tfrac{X_{S(\btheta)}}{n}  \le \sqrt{2\tau^*\log(|\mathcal{S}|)} \right) \le \Phi\left(\sqrt{2\tau^*\log(|\mathcal{S}|)} - A n \right).
	\end{equation}
	Now by \eqref{eq:mathcalSA} and since $\tau^\ast = \frac{(2 + \delta)^2}{4(1 + \delta)}$ (see \eqref{def:Tscan}), we can write
	\begin{equation}\label{eq:tau*ANm}
		\begin{split}
			&\sqrt{2\tau^*\log(|\mathcal{S}|)} - An = -(1 + o(1))\left(\sqrt{2(1 + \delta)} - \frac{2 + \delta}{\sqrt{2(1 + \delta)}}\right) \sqrt{\log(|\mathcal{S}|)}\\
			=& -(1 + o(1)) \frac{\delta}{\sqrt{2(1 + \delta)}} \sqrt{\log(|\mathcal{S}|)} = -(1 + o(1)) \sqrt{2(\tau^\ast - 1)} \sqrt{\log(|\mathcal{S}|)}.
		\end{split}
	\end{equation}
	Plugging this into \eqref{eq:scan_upperP2} and using \eqref{eq:normal_tailbnd} 
	again, we obtain
	\begin{equation*}
		\P_{\btheta}(T = 0) \le e^{-(1 + o(1))(\tau^\ast - 1)\log |\mathcal S|}
	\end{equation*}
	for any $\btheta \in \Theta(A)$. Together with \eqref{eq:scan_upperP0} (recall the 
	definition of $\tau^\ast$ from \eqref{def:Tscan}), this implies
	\begin{equation}\label{eq:scan_upper}
		\limsup_{n\rightarrow \infty}\frac{\log{{\mathrm{Risk}(T,A,n)}}}{\log{|\mathcal{S}|}} \le  -(\tau^*-1) = -\frac{\delta^2}{4(1+\delta)}
	\end{equation}
	which is the required upper bound in Theorem~\ref{thm:sparse_riskbnd}, part 1.
	
	\medskip

	\noindent{{\em Proof of the lower bound.}} Like in the previous subsection, we will apply Proposition~\ref{prop:risk_lowerbnd_ab} with $\bp = \bpi$. As to the choice of $\bE$, \eqref{eq:scan_upperP2} suggets that we should pick $$E_S = \big\{ \tfrac{X_S}{n} \le \sqrt{2\tau^*\log(|\mathcal{S}|)}\big\}$$ for all $S \in \mathcal S$ (cf.~Remark~\ref{rm:bE}). Let us first check the first moment of $L_{\bpi}^{\bE}$ corresponding to this choice.
	\begin{equation*}
		\begin{split}
			\E_{0}( L_{\bpi}^{\bE}) &\stackrel{\eqref{rewrit:ELpiE},\, \eqref{def:Lpidis}}{=} \frac{1}{|\mathcal S|}\sum_{S \in \mathcal S} \P_{\btheta(S)}\big(\tfrac{X_S}{n} \le \sqrt{2\tau^*\log(|\mathcal{S}|)}\,\big) \stackrel{\eqref{def:Xdist}}{=} \P_0\big( \tfrac{X_{[n]\times[n]}}{n} \le \sqrt{2\tau^*\log(|\mathcal{S}|)} - An\big) \\
			&\stackrel{\eqref{eq:mnregime}, \, \eqref{eq:tau*ANm}}{=} \Phi\big(-(1 + o(1)) \sqrt{2(\tau^\ast - 1)} \sqrt{\log(|\mathcal{S}|)}\,\big) \stackrel{\eqref{eq:normal_tailbnd}, \, (\tau^\ast = \frac{(2 + \delta)^2}{4(1 + \delta)})}{\ge} e^{-(1 + o(1)) \frac{\delta^2}{4(1 + \delta)}\log |\mathcal S|}.
		\end{split}
	\end{equation*}
	For the second term in \eqref{eq:risk_lowerbnd_ab} from 
	Proposition~\ref{prop:risk_lowerbnd_ab}, we would like to show
	\begin{equation}\label{eq:truncatedLpisq_spars_ab}
		\E_{0}\left( (  L_{\bpi}^{\bE} )^2 \right) \le e^{- (1 + o(1)) (\tau^\ast - 1) \log |\mathcal S|} = e^{-(1 + o(1)) \frac{\delta^2}{4(1 + \delta)}\log |\mathcal S|}
	\end{equation}
	so that one obtains by plugging the previous two bounds into \eqref{eq:risk_lowerbnd_ab}:
	\begin{equation}\label{eq:scan_lower}
		\liminf_{n\rightarrow \infty}\frac{\log{\mathcal R_n(A)}}{\log{|\mathcal{S}|}} \ge -\frac{\delta^2}{4(1+\delta)}
	\end{equation}
	which is the desired lower bound in Theorem~\ref{thm:sparse_riskbnd}. In the remainder of the proof for this part, we will derive \eqref{eq:truncatedLpisq_spars_ab}.

	Let us start by evaluating the second moment.
	\begin{equation}\label{eq:LpiEexpr}
		\begin{split}
			&\E_0((L_{\bpi}^{\bE})^2) \stackrel{\eqref{def:Lpidis}}{=} \frac 1{|\mathcal S|^2} \sum_{S, S' \in \mathcal S} \E_0\big(e^{A(X_S + X_{S'}) - A^2n^2} \, \mathbf 1_{\{\frac{X_{S}}n \vee \frac{X_{S'}}n \le  \sqrt{2\tau^*\log(|\mathcal{S}|)}\}} \big) \\
			&\stackrel{{\rm Lem}~\ref{lem:CM}, \, \eqref{eq:tau*ANm}}{=} \frac 1{|\mathcal S|^2} \sum_{S, S' \in \mathcal S} e^{A^2|S \cap S'|} \, p_{S, S'}
		\end{split}
	\end{equation}
	where 
	$$p_{S, S'} = \P_0 \big( \tfrac{X_{S}}{n} \vee \tfrac{X_{S'}}n \le  -(1 + o(1)) \sqrt{2(\tau^\ast - 1) \log(|\mathcal{S}|)} - A \tfrac{|S \cap S'|}{n}\big).$$
	By Lemma~\ref{lem:bvn}, we have
	\begin{equation*}
		p_{|S \cap S'|} :=  p_{S, S'} \le \exp\Big(-(1 + o(1))\tfrac{(\sqrt{2(\tau^\ast - 1)}\sqrt{\log |\mathcal S|} + A\frac{|S \cap S'|}{n})^2}{1 + \frac{|S \cap S'|}{n^2}}\Big).
	\end{equation*}
	Plugging this estimate into the right-hand side \eqref{eq:LpiEexpr} allows us to write, 
	with $k, \ell$ representing the number of common rows and columns between $S$ and 
	$S'$ respectively,
	\begin{equation}\label{eq:secmomexp}
		\E_0((L_{\bpi}^{\bE})^2)  \le \sum_{0 \le k, \ell \le n} \frac{N_{k, \ell}}{|\mathcal S|} \, e^{A^2k\ell}  \exp\Big(-(1 + o(1))\tfrac{(\sqrt{2(\tau^\ast - 1)}\sqrt{\log |\mathcal S|} + A\frac{k\ell}{n})^2}{1 + \frac{k\ell}{n^2}}\Big)
	\end{equation}
	where $N_{k, \ell}$ is as in \eqref{def:Nkl}. Now Corollary~\ref{cor:factorial} (recall that $N = \lfloor n^{1 + \alpha}\rfloor$) gives us
	\begin{equation*}
		\frac{N_{k, \ell}}{|\mathcal S|}\le  e^{C n} e^{- (k + \ell)\alpha \log n}  \le e^{C n} e^{- 2\sqrt{\frac{k\ell}{n^2}}\,\alpha n \log n} \stackrel{\eqref{eq:Scardinality}}{\le} e^{C n}e^{-\sqrt{\frac{k\ell}{n^2}}\log |\mathcal S|}
	\end{equation*}
	where in the final step we used the standard inequality $k + \ell \ge 2\sqrt{k\ell}$. Using 
	this bound and recalling \eqref{eq:mathcalSA}, we get from \eqref{eq:secmomexp}:
	\begin{equation*}
		\begin{split}
			\E_0((L_{\bpi}^{\bE})^2) \le & \: e^{Cn} n^2 \exp\Big( - (1 + o(1))\log |\mathcal S|  \inf_{0 \le k\ell \le n^2} \Big(  \tfrac{(\sqrt{2(\tau^\ast - 1)} + \sqrt{2(1 + \delta)}\frac{k\ell}{n^2})^2}{1 + \frac{k\ell}{n^2}}  + \sqrt{\tfrac{k\ell}{n^2}} - 2(1 + \delta)\tfrac{k\ell}{n^2} \Big) \Big).
		\end{split}	
	\end{equation*}
	Using the transformation of variables $k\ell \to \sqrt{\tfrac{k\ell}{n^2}} = q \in [0, 1]$ and 
	recalling from \eqref{def:Tscan} that $\tau^\ast - 1 = \frac{\delta}{2\sqrt{(1 + \delta)}}$, 
	the last display simplifies to
	\begin{equation}\label{eq:free_energy2}
		\begin{split}
			& \E_0((L_{\bpi}^{\bE})^2) \\
			\le& \:\, e^{Cn} n^2 \exp\Big( - (1 + o(1))\log |\mathcal S|  \inf_{0 \le q \le 1} \big( \, (\tau^\ast - 1) + (\tau^\ast - 1)\tfrac{1 - q^2}{1 + q^2} + q (1 - \tfrac{2q}{1 + q^2}) \, \big) \Big)\\
			\stackrel{(\tau^\ast - 1 \, \ge \, 0)}{ \le} & \:\, e^{Cn} n^2 \exp\big( - (1 + o(1)) (\tau^\ast - 1) \log |\mathcal S|  \big)
		\end{split}	
	\end{equation}
	which is precisely \eqref{eq:truncatedLpisq_spars_ab}.

	\smallskip

	We now show how to adapt the above argument in order to deduce part~2 of Theorem~\ref{thm:sparse_riskbnd}.
	
	\vspace{0.2cm}
	
	\noindent{\em Proof of Theorem~\ref{thm:sparse_riskbnd}, part~2.} First, note that $A = A^{\ast}_{\bullet}(\delta, \alpha) = n^{-(1 - \alpha - \delta)}$ satisfies 
	$$\frac{A^2 n^2}{2\log |\mathcal S|} = \frac{n^{2(\alpha + \delta)}}{2 \log |\mathcal S|} > 1$$
	for all $n > C(\alpha, \delta)$ since $\alpha + \delta > \frac12$ and $\log |\mathcal S| = (1 + o(1)) 2\alpha n \log n$. Thus there exists a unique $\tau_n^\ast \in (1, \infty)$ satisfying
	\begin{equation}\label{def:taunast}
		\sqrt{\tau_n^\ast} + 	\sqrt{\tau_n^\ast - 1} = \frac{An}{\sqrt{2 \log |\mathcal S|}}
	\end{equation}
	(cf.~\eqref{eq:tau*ANm}). Further, $\tau_n^\ast \to \infty$ as $n \to \infty$. Now replacing $\tau^\ast$ with $\tau_n^\ast$ in the proof of the previous 
	part, we obtain
	\begin{equation*}
		\lim_{n\rightarrow \infty}\frac{\log{\mathcal R_n(A)}}{(\tau_n^\ast - 1) \log{|\mathcal{S}|}} = -1.
	\end{equation*}
	(cf.~\eqref{eq:scan_upper}). This is equivalent to our desired limit in \eqref{eq:dense_riskbnd2} as $$(\tau_n^\ast - 1)\log |\mathcal S| = (1 + o(1)) \frac{A^2n^2}{8} = (1 + o(1))\frac{n^{2(\alpha + \delta)}}{8}$$ in view of \eqref{def:taunast}. \qed
	\section{Proofs in the below-boundary regime}\label{sec:below}
	In this section, we prove the results for the below-boundary regime $\delta < 0$, namely Theorems~\ref{thm:dense_below_boundary} and \ref{thm:sparse_below_boundary}. 
	Proposition~\ref{prop:risk_lowerbnd_bb} --- as opposed to 
	Proposition~\ref{prop:risk_lowerbnd_ab} --- is the key component in proving the lower bounds. Unlike the stretched and super-exponential rates 
	in the previous section which are contributed by a free energy-type single 
	dominant term in the squared truncated likelihood ratio, here we see 
	power law behaviors emerging from the agglomeration of several terms.

	\subsection{Proof of 
		Theorem~\ref{thm:dense_below_boundary}}\label{subsec:dense_below}
	Let us give the proof of the upper bound (on the risk) first.
	
	\medskip
	
	\noindent{{\em Proof of the upper bound.}} Taking cue from the proof of the 
	corresponding part for the above-boundary regime in Section~\ref{sec:dense_above}, 
	let us consider the sum test 
	\begin{equation*}
		T = T_{{\rm sum}}^0 = \mathbf{1}_{\{X_{[N]\times[N]} > 0\}}.
	\end{equation*}
	It is immediate that $\P_0(T=1)=\frac{1}{2}$. Moreover, for any $\btheta\in \Theta(A)$ with $A = A^\ast(\delta, \alpha) = n^{-(1 - \alpha - \delta)}$, one has (cf.~\eqref{eq:sumdense_upperP1})
	\begin{equation}\label{eq:Riskupperbb_dense}
		\P_{\btheta}(T=0) \leq \Phi(- n^{\delta}) \stackrel{(\delta < 0)}{\leq} \frac{1}{2} - c n^{\delta}.
	\end{equation}
	Combining the above two, we obtain
	\begin{equation*}
		\mathrm{Risk}(T, A, n) \le 1 - cn^{\delta}
	\end{equation*}
	which is the desired upper bound (on $\mathcal R_n(A)$) in 
	\eqref{eq:dense_below_boundary}.
	
	\medskip

	\noindent{{\em Proof of the lower bound.}} We will use 
	Proposition~\ref{prop:risk_lowerbnd_bb} {\em without} any truncation, i.e., each $E_S$ is the underlying full event. To this end let us 
	write (see, e.g., \eqref{eq:LpiEexpr}),
	\begin{equation*}
		\E_0\left(L_{\bpi}^2\right)=\E\left(\exp(A^2W_1W_2) \right),
	\end{equation*}
	where $W_1, W_2 \sim \mathrm{HG}(N, n, n)$ are independent 
	Hypergeometric variables under $\P$ (the corresponding expectation being $\E$). Now applying Lemma~\ref{lem:conv_order} for the exponential function 
	$\phi(x) = e^{A^2W_1 x}$ (conditionally on $W_1$), we can write
	\begin{equation*}
		\E_{\mathbf{0}}\left(L_{\bpi}^2\right) \le \E\left(\big(1 + \tfrac nN(e^{A^2W_1} - 1)\big)^n \right) \le \E\big(\exp\big(\tfrac{n^2}N (e^{A^2W_1} - 1) \big)\big)
	\end{equation*}
	where we used the standard inequality $1 + x \le e^x$ in the last step. Since $A^2 W_1 \le A^2 n = n^{-(1 - 2(\alpha + \delta))} \le 1$ (recall that $\alpha + \delta \le \frac 12$ and $A = n^{-(1 - \alpha - \delta)}$), we have $e^{A^2W_1} - 1 \le CA^2W_1$ and hence
	\begin{equation*}
		\E\big(\exp\big(\tfrac{n^2}N (e^{A^2W_1} - 1) \big)\big) \le \E\big(e^{C\frac{n^2}NA^2W_1}\big).
	\end{equation*}
	Reusing Lemma~\ref{lem:conv_order}, we thus obtain
	\begin{equation}\label{eq:Lpidense}
		\E_ 0\left(L_{\bpi}^2\right) \le \big(1+\tfrac{n}{N}(e^{C\frac{n^2}{N}A^2}-1)\big)^n \le \exp \big(  \tfrac{n^2}{N} (e^{C\frac{n^2}{N}A^2}-1) \big)
	\end{equation}
	where for the final step we used $1 + x \le e^x$. Again, owing to $\alpha + \delta \le 1/2$, we have $\frac {n^2}N A^2 = n^{2(\alpha + \delta) -1 - \alpha} \le 1$  and thus $e^{C\frac{n^2}{N}A^2} \le C\frac{n^2}{N}A^2$. 
	Consequently, the last inequality \eqref{eq:Lpidense} can be simplified to
	\begin{equation*}
		\E_ 0\left(L_{\bpi}^2\right) \le e^{C A^2\frac{n^4}{N^2}} \le e^{C n^{2\delta}}.
	\end{equation*}
	Now plugging this bound into \eqref{eq:risk_lowerbnd_bb} in 
	Proposition~\ref{prop:risk_lowerbnd_bb}, we get (recall that $\E_0(L_{\bpi}) =  1$):
	\begin{equation*}
		\mathcal R_n(A) \ge 1-\sqrt{\E_0\left(L_{\bpi}^2\right)-1} \ge 1 - \sqrt{e^{Cn^{2\delta}} - 1} \stackrel{(\delta < 0)}{\ge} 1 - c n^{\delta}
	\end{equation*}
	which is the required lower bound (on $\mathcal R_n(A)$) in \eqref{eq:dense_below_boundary}. Notice that the condition $\alpha + \delta \le 
	\frac 12$ is crucially used in the entire argument. \qed
	
	\subsection{Proof of Theorem~\ref{thm:sparse_below_boundary}}\label{subsec:sparse_below}The proof of Theorem~\ref{thm:sparse_below_boundary} has several surprising elements --- both technically and conceptually --- 
	and is the main highlight of this section.
	
	\medskip
	
	\noindent{\em Proof of the upper bound.} Our argument in the previous 
	subsection would suggest that a carefully designed scan test like what was 
	used for proving the corresponding upper bound in the above-boundary regime in \S\ref{subsec:sparse} should attain the optimal rates. Surprisingly, it turns out that the test considered for the dense region in \S\ref{subsec:dense_below}, i.e, $T {=} 
	\mathbf 1_{\{ X_{[N] \times [N]} > 0\}}$ is still the correct test (see 
	\S\ref{subsec:overview} for more detailed discussion). Indeed, using similar 
	arguments with $A = A^\ast_{{\rm S}}(\delta, \alpha) = \sqrt{4(1+\delta)\alpha\log n / n}$ (see around 
	\eqref{eq:Riskupperbb_dense}) we obtain
	\begin{equation}\label{eq:bb_sparseupper}
		\mathrm{Risk}(T,A,n) = \frac 12 + \Phi(-A\tfrac{n^2}N) \le 1 - c A\tfrac{n^2}N = 1 - c \sqrt{(1 + \delta)\alpha} \, n^{\frac12 - \alpha}\sqrt{\log n}\end{equation}
	where the bound in the second (equivalently, the third) step is valid since $N = \lfloor n^{1 + \alpha}\rfloor$ and $\alpha > \frac 12$. The above is the required upper bound (on $\mathcal R_n(A)$) in 
	\eqref{eq:sparse_below_boundary}.
	
	\medskip
	
	\noindent{\em Proof of the lower bound.} 
	This time we will use Proposition~\ref{prop:risk_lowerbnd_bb} {\em 
		with} truncation. In view of the test $T$ considered for the upper bound above (cf.~the proof of the lower bound in 
	Theorem~\ref{thm:dense_above_boundary}, part~1), one natural choice for the truncating event(s) $E_S$ would be $\{X_{[N]\times[N]} \le 
	B\}$ for some suitable $B$. However, one can check that {\em no} such choice would lead to the required bound. Instead, we truncate with respect to the events of the type $$E_S = \{ X_S \le B \}$$ 
	for $S \in \mathcal S$. We will specify $B$ after looking at the expression of $\E_0(L_{\bpi}^{\bE})$. To this end note that,
	\begin{equation*}
		\E_0(L_{\bpi}^{\bE}) \stackrel{{\rm Lem}~\ref{lem:CM}}{=} \frac 1{|\mathcal S|} \sum_{S \in \mathcal S}
		\P_{0}\big ( X_{S} + An^2  \le B\, \big) = 
		\Phi(\tfrac Bn - An).
	\end{equation*}
	Therefore choosing, e.g., $B = An^2 + 4n\sqrt{\alpha\log n}$, we find in view of 
	above and the tail bound from \eqref{eq:normal_tailbnd} that
	\begin{equation}\label{eq:expLpi_lowerbnd_ab}
	\E_0(L_{\bpi}^{\bE}) \ge 1 - n^{- 2 \alpha} \ge 1 - A \tfrac{n^2}{N}
	\end{equation}
	(cf.~\eqref{eq:bb_sparseupper}). In view of 
	Proposition~\ref{prop:risk_lowerbnd_bb} and \eqref{eq:expLpi_lowerbnd_ab}, it 
	therefore suffices to prove
	\begin{equation}\label{eq:second_mom_ab}
		\var_0(L_{\bpi}^{\bE}) \le C A^2\tfrac{n^4}{N^2} \mbox{ for all $n \ge C(\alpha, \delta)$.}
	\end{equation}
	
	\vspace{0.2cm}
	
	To this end, let us first evaluate $\E_0((L_{\bpi}^{\bE})^2)$. By \eqref{def:Lpidis} and Lemma~\ref{lem:CM}, we can write
	\begin{equation}\label{eq:ELpi'sq0}
		\E_0((L_{\bpi}^{\bE})^2) = \frac 1{|\mathcal S|^2} \sum_{S, S' \in \mathcal S} e^{A^2|S \cap S'|} \, \P_0 \big( X_{S} \vee X_{S'} \le  q - A|S \cap S'| \big)
	\end{equation}
	where 
	\begin{equation}\label{eq:range_q}
q = B - An^2 = 4 n \sqrt{\alpha \log n}.
	\end{equation}
 From Lemma~\ref{lem:bvn}, we obtain:
	\begin{equation}\label{eq:pSS'}
		\begin{split}
			\P_0 \big( X_{S} \vee X_{S'} \le  q - A|S \cap S'| \big) = p_{|S \cap S'|} 
			\le \exp\Big(- \tfrac{\big((A\frac{|S \cap S'|}{n} - \tfrac{q}{n} )^+\big)^2}{1 + \frac{|S \cap S'|}{n^2}}\Big)
		\end{split}
	\end{equation}
	(recall that $x^+ = \max(x, 0)$). In view of \eqref{eq:expLpi_lowerbnd_ab}, 
	\eqref{eq:ELpi'sq0} and \eqref{eq:pSS'}, we can write
	\begin{equation}\label{eq:var_bnd}
		\var_0(L_{\bpi}^{\bE}) \le 
		\sum_{1 \le k, \ell \le n} \frac{N_{k, \ell}}{|\mathcal S|} \Big(e^{A^2k\ell} \exp\Big(- \tfrac{\big((A\frac{|S \cap S'|}{n} - \tfrac{q}{n})^+\big)^2}{1 + \frac{|S \cap S'|}{n^2}}\Big) - (1 - 2n^{-2\alpha}) \Big)
	\end{equation}
	where $N_{k, l}$'s are as in \eqref{def:Nkl}. Denoting the sum on the 
	right-hand side above by $\Sigma$, we decompose it into three parts,
	\begin{equation}\label{eq:var_decomp}
		\Sigma = \Sigma_1 + \Sigma_2 + \Sigma_3
	\end{equation}
	where $\Sigma_1$, $\Sigma_2$ and $\Sigma_3$ restrict the indices $1 \le k, \ell \le n$ to the regimes $k\ell \le \tfrac 1 {A^2}$, $\tfrac 1 {A^2} < k\ell \le \tfrac {2q}A$ and $\tfrac {2q}A < k\ell \le n^{2}$ respectively. Since $A = \sqrt{\frac{4(1+\delta)\alpha\log n}{n}}$ and $q = 4n \sqrt{\alpha \log n}$ (see \eqref{eq:range_q}), these summation ranges are well-defined for 
	sufficiently large $n$ (depending on $\alpha$ and $\delta$). We will now 
	treat each of these sums separately.

	We will show that $\Sigma_2$ and $\Sigma_3$ decay stretched-exponentially fast and are thus of lower 
	order than the desired upper bound in \eqref{eq:second_mom_ab}. It is $\Sigma_1$ which governs the {\em power law} behavior in 
	\eqref{eq:second_mom_ab}. 
	Henceforth in this proof we will implicitly assume that all the statements hold for $n \ge C(\alpha, \delta)$. 
	
	\medskip
	
	\noindent{\bf Bounding $\Sigma_3$.} Since $A\tfrac{k\ell}n > 2\frac{q}n$, i.e., 
	$A\tfrac{k\ell}n - \frac{q}n > \frac{q}n (> 0)$, we get the following from 
	\eqref{eq:var_bnd}--\eqref{eq:var_decomp}, the upper bound on $N_{k, \ell}$ as given by \eqref{eq:factorial} (recall that $N = \lfloor n^{1 + \alpha} \rfloor$) 
	and the upper bound in \eqref{eq:normal_tailbnd}:
	\begin{equation}\label{eq:Sigma3bnd1}
		\Sigma_3 \le C n^2 \sup_{\substack{1 \le k, \ell \le n,\\ k\ell > \frac{2q}A}}\exp\big(A^2k\ell - \tfrac{(A\frac{k\ell}{n} - \tfrac qn)^2}{1 + \frac{k\ell}{n^2}} + k\log \tfrac{en^{1 - \alpha}}{k}  + \ell\log \tfrac{en^{1 - \alpha}}{\ell} \big). 
	\end{equation}
	We will need the following result.
	
	\vspace{0.3cm}
	
	\noindent{{\bf Claim.}} 
	Let $w, z \in (0, \infty)$ be such that $w > \tfrac{z^2}{e^2}$. Then the function $x \mapsto x \log \tfrac zx + x^\ast \log \tfrac {z}
	{x^\ast}$ (cf.~\eqref{eq:factorial}), where $x x^\ast = w$, achieves its maximum on $(0, \infty)$ at the point $x = \sqrt{w}$.
	
	\vspace{0.2cm}
	
	We postpone the proof of this Claim until the end and continue to the pending 
	bound on $\Sigma_3$. Note that any $k, \ell$ included under the summation in 
	\eqref{eq:Sigma3bnd1} satisfies
	\begin{equation}\label{eq:kell_bnd1}
		k\ell > 2 \frac q{A} \stackrel{\eqref{eq:range_q}}{\ge} c(\alpha) \frac{n\sqrt{\log n}}{A} = c(\alpha, \delta) n^{\frac32} \stackrel{(\alpha > \frac12)}{>} n^{2(1 - \alpha)} = \frac{(e n^{1 - \alpha})^2}{e^2}\end{equation}
	where in the third step we used $A = c(\alpha, \delta)\sqrt{\frac{\log n}{n}}$. 
	This enables us use to use the above Claim with $k, \ell$ and $k\ell$ as $x, x^\ast$ and $w$ respectively and $e n^{1 - \alpha}$ as $z$ for any such 
	values of $k$ and $\ell$. Consequently, applying the transformation $\rho = 
	\frac{\sqrt{k\ell}}n$, we can rewrite \eqref{eq:Sigma3bnd1} as follows:
	\begin{equation}\label{eq:Sigma3bnd2}
		\Sigma_3 \le C n^2 \sup_{c(\alpha, \delta)n^{-\frac14} \le \rho \le 1} \exp\big(A^2n^2 \rho^2 - \tfrac{(An\rho^2 - \tfrac qn)^2}{1 + \rho^2} + 2n \rho \log \tfrac{en^{-\alpha}}{\rho} \big ).
	\end{equation}
	Let us simplify the right-hand side before proceeding further. Recalling that $A = 
	\sqrt{\frac{4(1+\delta)\alpha\log n}n}$, we get the equivalent expression
	\begin{equation*}
		2n \rho \log \tfrac{en^{-\alpha}}{\rho} = -A^2n^2\rho(\tfrac 1{2(1 + \delta)} - \tfrac{\log \frac e{\rho}}{2(1 + \delta)\alpha \log n}).
	\end{equation*}
	On the other hand, since $An \cdot \frac qn \le C(\alpha, \delta)\sqrt{n}\log n$, 
	which follows from the choice of $A$ above as well as \eqref{eq:range_q}, 
	we have the upper bound
	\begin{equation*}
		A^2n^2\rho^2 - \tfrac{(An\rho^2 - \tfrac qn)^2}{1 + \rho^2} 
		\le A^2n^2\tfrac{\rho^2}{1 + \rho^2} + C(\alpha, \delta)\sqrt{n}\log n.
	\end{equation*}
	Now plugging the above two expressions into the right-hand side of 
	\eqref{eq:Sigma3bnd2}, we obtain
	\begin{equation*}
		\Sigma_3 \le C n^2 \sup_{c(\alpha, \delta)n^{-\frac14} \le \rho \le 1} \exp\big(-A^2n^2 \rho(\tfrac 1{2(1 + \delta)} - \tfrac{\log \frac e{\rho}}{2(1 + \delta)\alpha \log n} - \tfrac{\rho}{1 + \rho^2}) + C(\alpha, \delta)\sqrt{n}\log n\big).
	\end{equation*}
	Since $\delta \in (-1, 0)$ and $\alpha > \frac 12$, this yields
	\begin{equation}\label{eq:Sigma3bnd4}
		\Sigma_3 \le C n^2\exp(-c(\alpha, \delta) A^2n^2n^{-\frac14} + C(\alpha, \delta)\sqrt{n}\log n)  \le C(\alpha, \delta) \exp(-c(\alpha, \delta) n^{\frac34}\log n).
	\end{equation}
	
	\medskip
	
	\noindent{\bf Bounding $\Sigma_2$.} In this summation we can afford to ignore the 
	negative term on the right-hand side of \eqref{eq:Sigma3bnd1} and write,
	\begin{equation*}
		\Sigma_2 \le C n^2 \sup_{\substack{1 \le k, \ell \le n,\\ \frac 1{A^2} < k\ell \le \frac{2q}A}}\exp\big(A^2k\ell + k\log \tfrac{en^{1 - \alpha}}{k}  + \ell\log \tfrac{en^{1 - \alpha}}{\ell}\big).
	\end{equation*}
	Since $k\ell > \frac 1{A^2} \ge c(\alpha, \delta) \frac n{\log n} > n^{2(1 - \alpha)}$ (cf.~\eqref{eq:kell_bnd1}), we can again use the Claim below 
	\eqref{eq:Sigma3bnd1} to obtain (cf.~\eqref{eq:Sigma3bnd2}) 
	\begin{equation*}
		\Sigma_2 \le C n^2 \sup_{c(\alpha, \delta)(n\log n)^{-\frac12} \le \rho \le C(\alpha, \delta) n^{-\frac14}} \exp\big(A^2n^2 
		\rho^2 + 2n \rho \log \tfrac{en^{-\alpha}}{\rho} \big) 
	\end{equation*}
	Since $\alpha > \frac12$, we have $\log \tfrac{en^{-\alpha}}{\rho} \le -c(\alpha, \delta) \log n$ for $\rho > c(\alpha, \delta)(n\log n)^{-\frac12}$. Consequently, 
	\begin{equation*}
		\Sigma_2 \le C n^2 \sup_{c(\alpha, \delta)(n\log n)^{-\frac12} \le \rho \le C(\alpha, \delta) n^{-\frac14}} \exp\big(A^2n^2 
		\rho^2 - c(\alpha, \delta)n \rho \log n \big)
	\end{equation*}
	The supremum above is clearly achieved at the endpoints leading to
	\begin{equation}\label{eq:Sigma2bnd4}
		\Sigma_2 \le C(\alpha, \delta) \exp(-c(\alpha, \delta)\sqrt{n\log n}).
	\end{equation}
	
	\medskip
	
	\noindent{\bf Bounding $\Sigma_1$.} As already mentioned, this is the term 
	which determines the principal order of the variance of $L_{\bpi}^{\bE}$. We 
	can write, in view of \eqref{eq:var_bnd}--\eqref{eq:var_decomp},
	\begin{equation*}
	\begin{split}
	\Sigma_1 &\le \sum_{1 \le k\ell \le \frac1{A^2}} \frac{N_{k, \ell}}{|\mathcal S|} \big(e^{A^2k\ell}  - (1 - 2n^{-2\alpha}) \big) \le \sum_{1 \le k\ell \le \frac1{A^2}} \frac{N_{k, \ell}}{|\mathcal S|} (e^{A^2k\ell}  - 1) + 2n^{-2\alpha}\\
	&\le C \sum_{1 \le k\ell \le \frac1{A^2}} \frac{N_{k, \ell}}{|\mathcal S|} A^2 k \ell + 2n^{-2\alpha} \le C A^2 \sum_{0 \le k\ell \le n} \frac{N_{k, \ell}}{|\mathcal S|} \, k \ell + 2n^{-2\alpha} = C A^2 \E (W_1 W_2) + 2n^{-2\alpha}
	\end{split}
	\end{equation*}
	where, like in the previous subsection, $W_1, W_2 \sim \mathrm{HG}(N, n, n)$ are independent Hypergeometric variables under $\P$. Consequently,
    \begin{equation}\label{eq:Sigma1bnd2_4}
    	\Sigma_1 \le C A^2 \E (W_1)^2 + 2n^{-2\alpha} = C A^2 \frac{n^4}{N^2} + 2n^{-2\alpha} \le C A^2 \frac{n^4}{N^2}
    \end{equation}
	where in the last step we used $A \frac{n^2}{N} = c(\alpha, \delta)n^{\frac12 - \alpha} \sqrt{\log n}$ (cf.~\eqref{eq:bb_sparseupper}).

	\medskip
	
	Combining the bounds on $\Sigma_1$, $\Sigma_2$ and $\Sigma_3$, we obtain from \eqref{eq:var_bnd} and \eqref{eq:var_decomp}:
	\begin{equation*}
		\var(L_{\bpi}^{\bE}) \le  \Sigma_1 + \Sigma_2 + \Sigma_3 \stackrel{\eqref{eq:Sigma3bnd4}, \eqref{eq:Sigma2bnd4}, 
			\eqref{eq:Sigma1bnd2_4}}{\le} CA^2 \tfrac{n^4}{N^2}
	\end{equation*}
	which is precisely \eqref{eq:second_mom_ab}, thus completing the proof. It 
	only remains to give the
	
	\medskip
	
	\noindent{{\em Proof of the Claim.}} Let us consider the function $f: [0, \infty)^2 \to [0, \infty)$ defined as $f(x, y) = x \log \tfrac zx + y \log \tfrac {z}{y}$. $f$ is 
	continuous and is differentiable on $(0, \infty)^2$. Furthermore, since $f(x, y)$ tends to $-\infty$ as $x \to 0$ and $y \to \infty$ (or 
	vice versa), we obtain that the function $x \mapsto x \log \tfrac zx + x^\ast \log \tfrac {z}{x^\ast}$, with $xx^\ast = w$, attains its maximum at some point $x \in 
	(0, \infty)$. Therefore using the method of Lagrange multipliers, we obtain that 
	$x$ necessarily satisfies
	\begin{equation}\label{eq:zxl}
		\log \tfrac z{ex} = \lambda x^\ast \mbox{ and } \log \tfrac z{ex^\ast} = \lambda x
	\end{equation}
	for some $\lambda \in \R$. Consequently, we have $x \log \tfrac z{ex}  = x^\ast \log \tfrac z{ex^\ast}$. However, the function $x \mapsto x \log \tfrac z{ex}$ is strictly monotone on either side of $\tfrac z{e^2}$. Therefore for \eqref{eq:zxl} to hold, we need either $x = x^\ast = \sqrt{w}$ or $x, x^\ast$ to lie on 
	the opposite sides of $\tfrac z{e^2}$. Now since $xx^\ast = w > \tfrac{z^2}{e^2}$, 
	at least one of $x$ and $x^\ast$ must be larger than $\tfrac{z}e$. Consequently, if 
	$x \ne x^\ast$, we must have $\log \tfrac z{ex} \log \tfrac z{ex^\ast} < 0$ which 
	contradicts \eqref{eq:zxl} thereby yielding the Claim. \qed

	\section{Proof in the critical regime}\label{sec:criticality} 
	The sole result that we prove in this section is Theorem~\ref{thm:criticality}. 
	Like with the previous results, we first give the
	
	\medskip
	
	\noindent{\em Proof of the upper bound.} We use the following scan test (cf.~\eqref{def:Tscan}):
	\begin{equation*}
		T = T_{{\rm scan}}^{2\sqrt{n \log N}} = \mathbf 1_{\big\{\max_{S \in \mathcal S} \frac {X_S}{n} > 2\sqrt{n \log N}\big\}}.
	\end{equation*}
	Since each $\frac{X_S}n$ is distributed as $N(0, 1)$ under $\P_0$, we can use the upper tail bound in \eqref{eq:normal_tailbnd} and a 
	union bound to deduce
	\begin{equation*}
		\P_0(T = 1) \le N^{2n} (1 - \Phi(2 \sqrt{n \log N})) \le \frac C {\sqrt{n \log N}}
	\end{equation*}
	(for $n, N > 1$). As to the type-II error, let us note that $\frac{X_S(\btheta)}n$ 
	is distributed as $N(An, 1)$ under $\P_{\btheta}$ for any $\btheta \in \Theta(A)$ (cf.~\eqref{eq:scan_upperP2}) and hence with $A = A_{\rm c}^\ast = 2\sqrt{\frac{\log N}n}$,
	\begin{equation*}
		\P_{\btheta}(T = 0) \le \Phi(2\sqrt{n\log N} - An) = \Phi(0) =  \frac12
	\end{equation*}
	for all $\btheta \in \Theta(A)$. Together, the last two displays yield the upper bound in \eqref{eq:lim_critical}.
	
	\medskip
	
	\noindent{\em Proof of the lower bound.} We will use 
	Proposition~\ref{prop:risk_lowerbnd_bb} for deriving the lower bound. Like we have done --- for the most part --- in the previous sections, 
	we can truncate with respect to events corresponding to the test statistic 
	leading to the upper bound. More precisely, we work with the events
	\begin{equation*}
		E_S = \{ \tfrac{X_S}n \le An \}
	\end{equation*}
	for $S \in \mathcal S$ (cf.~the choice in \S\ref{subsec:sparse} and also note that $An = 2 \sqrt{n \log n}$). By Proposition~\ref{prop:risk_lowerbnd_bb}, 
	we have
	\begin{equation*}
		\mathcal R_n(A) \ge \E_0(L_{\bpi}^{\bE}) - \frac12\sqrt{\var_0(L_{\bpi}^{\bE})}.
	\end{equation*}
	Using similar argument as in \eqref{eq:expLpi_lowerbnd_ab}, we get $\E_0(L_{\bpi}^{\bE}) = 
	\frac 12$. Therefore, all we need to show is that $\lim_{N \to \infty} \var_0(L_{\bpi}^{\bE}) = 0$. To this end, we can write similarly as in \eqref{eq:ELpi'sq0}:
	\begin{equation*}
		\var_0(L_{\bpi}^{\bE}) \le \E_0((L_{\bpi}^{\bE})^2) 
		= \frac 1{|\mathcal S|^2} \sum_{S, S' \in \mathcal S} e^{A^2|S \cap S'|} \, \P_0 \big( X_{S} \vee X_{S'} \le  - A|S \cap S'| \big).
	\end{equation*}
	We bound these probabilities by \eqref{eq:pSS'} (with $0$ in place of $q$) for 
	all values of $|S \cap S'|$ {\em except} when $|S \cap S'| = n^2$ (i.e., $S = S'$) in 
	which case we can use the more direct (and exact) bound, namely (recall that $X_S$ is 
	distributed as $N(0, n^2)$ under $\P_0$)
	\begin{equation*}
		\P_0 \big(X_{S} \le  - An^2 \big) \stackrel{\eqref{eq:normal_tailbnd}}{\le} \frac 1{An} e^{-\frac{A^2n^2}2}.
	\end{equation*}
	All in all we obtain, as in \eqref{eq:var_bnd},
	\begin{equation*}
		\var_0(L_{\bpi}^{\bE}) \le C \sum_{\substack{1 \le k, \ell \le n,\\ k\ell \ne n^2}} \frac{N_{k, \ell}}{|\mathcal S|} e^{A^2k\ell} \exp\big(- \tfrac{A^2(k\ell)^2}{n^2 + k\ell}\big) + \frac1{|\mathcal S|} e^{A^2 n^2} \frac 1{An} e^{-\frac{A^2n^2}2}
	\end{equation*}
	where $N_{k, l}$ is as in \eqref{def:Nkl} and we already used the upper tail bound for $\Phi(-\tfrac{A k\ell}n)$ from \eqref{eq:normal_tailbnd} in the first 
	summation. Since $\frac{N_{k, \ell}}{|\mathcal S|} \le C(n) N^{- k - \ell}$ and $A = 2\sqrt{\frac{\log 
			N}n}$, we can rewrite 
	the above as
	\begin{equation}\label{eq:var_bndcrit}
		\var_0(L_{\bpi}^{\bE}) \le C(n) \sum_{\substack{1 \le k, \ell \le n,\\ k\ell \ne n^2}} \exp\big(- \log N (\tfrac{4(k\ell)^2}{n(n^2 + k\ell)} + k + \ell  - 4\tfrac{k\ell}n)\big) + \frac 1{\sqrt{n \log N}}.
	\end{equation}
	Now using the standard inequality $k + \ell \ge 2\sqrt{k \ell}$ and a 
	little algebra we obtain
	\begin{equation*}
		\begin{split}
			\tfrac{4(k\ell)^2}{n(n^2 + k\ell)} + k + \ell  - 4\tfrac{k\ell}n \ge \tfrac{4(k\ell)^2}{n(n^2 + k\ell)} + 2\sqrt{k\ell}  - 4\tfrac{k\ell}n = \tfrac{2\sqrt{k\ell}}{n^2 + k\ell} (n - \sqrt{k \ell})^2.
		\end{split}
	\end{equation*}
	This term is bounded below by $c(n) (> 0) $ when $k\ell < n^2$, i.e., $k\ell 
	\le n(n-1)$. Plugging this bound into the right-hand side of \eqref{eq:var_bndcrit} we get $\lim_{N \to \infty} \var_0(L_{\bpi}^{\bE}) = 0$ finishing the proof. \qed

	\bigskip
	
	\noindent{\bf Acknowledgement.} SG's research is supported by the ANRF grant ANRF/ARGM/2025/001947/MTR and a grant from the Department of Atomic Energy, 
	Government of India, under Project Identification No.~RTI4014. The authors sincerely thank Cristina Butucea and Subhabrata Sen for many conversations at the beginning of the project that helped them deeply in shaping the paper. The authors acknowledge the hospitality of the Statistics and Mathematics Unit, Indian Statistical Institute (ISI), Kolkata on several occasions. RM thanks the School of Mathematics, Tata Institute of Fundamental Research (TIFR), Mumbai for their generous hospitality.
	
	\bigskip
	
	\noindent{\bf AI use Statement.} No mathematical content was generated or found with the assistance of artificial intelligence.

	\bibliographystyle{apacite}
	
\end{document}